\documentclass{amsart}

\usepackage{amsmath}             
\usepackage{amsfonts}             
\usepackage{amsthm}               
\usepackage{amsbsy}
\usepackage{amssymb}
\usepackage{amsthm}

\newtheorem{thm}{Theorem}[section]
\newtheorem{lem}[thm]{Lemma}
\newtheorem{prop}[thm]{Proposition}
\newtheorem{cor}[thm]{Corollary}

\theoremstyle{definition}
\newtheorem{defn}[thm]{Definition}

\newtheorem{exam}[thm]{Example}

\begin{document}

\nocite{*}

\title[Closed Unitary and Similarity Orbits in C$^*$-Algebras]{Closed Unitary and Similarity Orbits of Normal Operators in Purely Infinite C$^*$-Algebras}

\author{Paul Skoufranis}
\address{Department of Mathematics, UCLA, Los Angeles, California, 90095-1555, USA}
\email{pskoufra@math.ucla.edu}
\thanks{This research was supported in part by NSERC PGS D3-389187-200.}
\subjclass[2010]{Primary: 46L05; Secondary:  47A99}
\date{November 29, 2012.}
\keywords{purely infinite C$^*$-algebra, closed unitary orbits, closed similarity orbit, normal operator}

\begin{abstract}
We will investigate the norm closure of the unitary and similarity orbits of normal operators in unital, simple, purely infinite C$^*$-algebras.  An operator theoretic proof will be given to the classification of when two normal operators are approximately unitarily equivalent in said algebras with trivial $K_1$-group.  Some upper and lower bounds for the distance between unitary orbits will be obtained based on these methods.  In addition, a complete characterization of when one normal operator is in the closed similarity orbit of another normal operator will be given for unital, simple, purely infinite C$^*$-algebras and type III factors with separable predual.
\end{abstract}

\maketitle

\maketitle

\section{Introduction}
\label{sec:INTRO}

Significant research has been performed in determining when two normal operators in a unital C$^*$-algebra are approximately unitarily equivalent.  For example the Weyl-von Neumann-Berg Theorem determines when two normal operators in the bounded linear maps on a complex, separable, infinite dimensional Hilbert space are approximately unitarily equivalent (see {\cite[Theorem II.4.4]{Da4}} for example) and a famous paper due to Brown, Douglas, and Fillmore can be used to determine when two normal operators in the Calkin algebra are approximately unitarily equivalent (see {\cite[Theorem 11.1]{BDF}}).  More recently {\cite[Theorem 1.3]{Sh}} completely determines when two normal operators in a von Neumann algebra of an arbitrary single type are approximately unitarily equivalent.
\par
Given a normal operator $N$ in a unital C$^*$-algebra $\mathfrak{A}$, the Continuous Functional Calculus for Normal Operators provides a unital, injective $^*$-homomorphism from the continuous functions on the spectrum of $N$ into $\mathfrak{A}$ sending the identity function to $N$.  It is easy to see that two normal operators are approximately unitarily equivalent in $\mathfrak{A}$ if and only if the corresponding unital, injective $^*$-homomorphism are approximately unitarily equivalent.  Thus it is of interest to determine when two unital, injective $^*$-homomorphisms from an abelian C$^*$-algebra to a fixed unital C$^*$-algebra are approximately unitarily equivalent.  In particular, a complete classification was given in \cite{Dad} for unital, simple, purely infinite C$^*$-algebras.
\begin{thm}[{\cite[Theorem 1.7]{Dad}}]
\label{dadaresult}
Let $X$ be a compact metric space, let $\mathfrak{A}$ be a unital, simple, purely infinite C$^*$-algebra, and let $\varphi, \psi : C(X) \to \mathfrak{A}$ be two unital, injective $^*$-homomorphisms.  Then $\varphi$ and $\psi$ are approximately unitarily equivalent if and only if $[[\varphi]] = [[\psi]]$ in $KL(C(X), \mathfrak{A})$ (see \cite{Ro} for the definition of $KL$).
\end{thm}
As a specific case of {\cite[Theorem 1.7]{Dad}}, if $X \subseteq \mathbb{C}$ is compact, it is a corollary of the Universal Coefficient Theorem for C$^*$-algebras (see \cite{RS}), the definition of $KL(C(X), \mathfrak{A})$, and the fact that $K_*(C(X))$ is a free abelian group that 
\[
KL(C(X), \mathfrak{A}) = KK(C(X), \mathfrak{A}) = Hom(K_*(C(X)), K_*(\mathfrak{A}))
\]
where $Hom(K_*(C(X)), K_*(\mathfrak{A}))$ is the set of all homomorphisms from $K_*(C(X))$ to $K_*(\mathfrak{A})$.  Thus {\cite[Theorem 1.7]{Dad}} implies that for a unital, simple, purely infinite C$^*$-algebra $\mathfrak{A}$ and a compact subset $X$ of $\mathbb{C}$, two unital, injective $^*$-homomorphisms $\varphi, \psi : C(X) \to \mathfrak{A}$ are approximately unitarily equivalent if and only if $\varphi_* = \psi_*$ where $\varphi_*$ and $\psi_*$ are the group homomorphisms from $K_*(C(X))$ to $K_*(\mathfrak{A})$ induced by $\varphi$ and $\psi$ respectively.  Thus a complete classification of when two normal operator with the same spectrum in a unital, simple, purely infinite C$^*$-algebra is obtained.  
\par
The proof of Dadarlat's result greatly varies from the traditional proof of when two normal operators on a complex, infinite dimensional, separable Hilbert space are approximately unitarily equivalent.  In Section \ref{sec:CLOSEDUNITARYORBITS} we shall give an operator theoretic proof of the classification of when two normal operators are approximately unitarily equivalent in a unital, simple, purely infinite C$^*$-algebra with trivial $K_1$-group.  Although the results of Section \ref{sec:CLOSEDUNITARYORBITS} are less powerful than {\cite[Theorem 1.7]{Dad}}, the more traditional nature of the proof enables the study of additional operator theoretic problems on these C$^*$-algebras.
\par
One particularly interesting problem is the study of the distance between unitary orbits of operators.  Significant progress has been made in determining the distance between two unitary orbits of bounded operators on a complex, infinite dimensional Hilbert space (see \cite{Da2} and \cite{Da3}).  In terms of determining the distance between unitary orbits of normal operators inside other C$^*$-algebras, \cite{Da1} makes significant progress for the Calkin algebra (which is a unital, simple, purely infinite C$^*$-algebra) and \cite{HN} makes significant progress for semifinite factors.  
\par
In Section \ref{sec:DISTANCEUNITARYORBITS} we will make use of the approach of Section \ref{sec:CLOSEDUNITARYORBITS} to compute some bounds on the distance between unitary orbits of normal operators in unital, simple, purely infinite C$^*$-algebras with trivial $K_1$-group.  Using {\cite[Theorem 1.7]{Dad}} along with some additional K-theory arguments, we will extend these results to unital, simple, purely infinite C$^*$-algebras without any constraints on the $K_1$-groups.
\par
Another interesting operator theoretic problem is describing the norm closure of the similarity orbit of a given operator.  A complete classification of the closed similarity orbit of an arbitrary bounded linear operator on a complex, infinite dimensional Hilbert space was announced in {\cite[Theorem 1]{AHV}} and a proof was given in {\cite[Theorem 9.2]{AFHV}}.  An easy modification of the proof of {\cite[Theorem 1]{AHV}} led to a complete classification of the closed similarity orbit of an arbitrary operator in the Calkin algebra (announced in {\cite[Theorem 2]{AHV}} and proved in {\cite[Theorem 9.3]{AFHV}}).  
\par
Before {\cite[Theorem 1]{AHV}} a classification of when one normal operator on a complex, infinite dimensional, separable Hilbert space was in the closed similarity orbit of another operator with minor additional constraints was given in {\cite[Theorem 1]{BH}}.  Thus it appears natural when tackling the problem of computing the norm closure of the similarity orbit of an operator in a unital C$^*$-algebra to first consider the normal operators.  Using the results from Section \ref{sec:DISTANCEUNITARYORBITS} along with a result from \cite{Sk} and ideas from \cite{He}, a classification of when one normal operator is in the closed unitary orbit of another normal operator in unital, simple, purely infinite C$^*$-algebras and type III factors with separable predual will be given in Section \ref{sec:CLOSESIMORBITS}.

\section{Closed Unitary Orbits of Normal Operators}
\label{sec:CLOSEDUNITARYORBITS}

In this section we will provide an operator theoretic proof of when two normal operators in a unital, simple, purely infinite C$^*$-algebra with trivial $K_1$-group are approximately unitarily equivalent (see Corollary \ref{aueintrivialk1}).  Along the way we shall develop the notation and several technical results that will necessary in later sections and develop analogous results for other C$^*$-algebras.
\par
For the discussions in this paper, $\mathfrak{A}$ will denote a unital C$^*$-algebra, $\mathcal{U}(\mathfrak{A})$ will denote the unitary group of $\mathfrak{A}$, $\mathfrak{A}^{-1}$ will denote the group of invertible elements of $\mathfrak{A}$, and $\mathfrak{A}^{-1}_0$ will denote the connected component of the identity in $\mathfrak{A}^{-1}$.  For a fixed C$^*$-algebra $\mathfrak{A}$ and an operator $A \in \mathfrak{A}$, let $\sigma(A)$ denote the spectrum of $A$ in $\mathfrak{A}$, let
\[
\mathcal{U}(A) := \{UAU^* \in \mathfrak{A} \, \mid \, U \in \mathcal{U}(\mathfrak{A})\},
\]
and let
\[
\mathcal{S}(A) := \left\{VAV^{-1} \in \mathfrak{A} \, \mid \, V \in \mathfrak{A}^{-1}\right\}.
\]
The set $\mathcal{U}(A)$ is called the unitary orbit of $A$ in $\mathfrak{A}$ and $\mathcal{S}(A)$ is called the similarity orbit of $A$ in $\mathfrak{A}$.
\par
Notice if $B \in \mathfrak{A}$ then $B \in \mathcal{U}(A)$ if and only if $A \in \mathcal{U}(B)$ and $B \in \mathcal{S}(A)$ if and only if $A \in \mathcal{S}(B)$.  We will denote $B \in \mathcal{U}(A)$ by $A \sim_u B$ and we will denote $B \in \mathcal{S}(A)$ by $A \sim B$.  Clearly $\sim_u$ and $\sim$ are equivalence relations.
\par
We will use $\overline{\mathcal{U}(A)}$ and $\overline{\mathcal{S}(A)}$ to denote the norm closures in $\mathfrak{A}$ of the unitary and similarity orbits of $A$ respectively.  Note if $B \in \overline{\mathcal{U}(A)}$ then $A \in \overline{\mathcal{U}(B)}$ and $B \in \overline{\mathcal{S}(A)}$.  If $B \in \overline{\mathcal{U}(A)}$ we will say that $A$ and $B$ are approximately unitarily equivalent in $\mathfrak{A}$ and will write $A \sim_{au} B$.  Clearly $\sim_{au}$ is an equivalence relation.  Furthermore if $A$ is a normal operator and $A \sim_{au} B$ then $B$ is a normal operator.  If $B \in \overline{\mathcal{S}(A)}$ then it is not necessary that $A \in \overline{\mathcal{S}(B)}$ and $B$ need not be normal if $A$ is normal.  However if $B \in \overline{\mathcal{S}(A)}$ and $C \in \overline{\mathcal{S}(B)}$ then $C \in \overline{\mathcal{S}(A)}$.
\par
It is an easy application of the semicontinuity of the spectrum to show that if $A, B \in \mathfrak{A}$ are such that $B \in \overline{\mathcal{S}(A)}$ then $\sigma(A) \subseteq \sigma(B)$ and $\sigma(A)$ intersects every connected component of $\sigma(B)$.  Thus $\sigma(A) = \sigma(B)$ whenever $A, B \in \mathfrak{A}$ are approximately unitarily equivalent.
\par
It is useful for discussions in this paper to recall the generalized index function introduced in \cite{Li}.
\begin{defn}
Let $\mathfrak{A}$ be a unital C$^*$-algebra and let $N \in \mathfrak{A}$ be a normal operator.  By the Continuous Functional Calculus for Normal Operators, there exists a canonical unital, injective $^*$-homomorphism $\varphi_N : C(\sigma(N)) \to \mathfrak{A}$ such that $\varphi_N(z) = N$.  As $\varphi_N$ is unital and injective, this induces a group homomorphism $\Gamma(N) : K_1(C(\sigma(N))) \to K_1(\mathfrak{A})$.  The group homomorphism $\Gamma(N)$ is called the index function of $N$.  To simplify notation, we will write $\Gamma(N)(\lambda)$ to denote $[\lambda I_\mathfrak{A} - N]_1$ in $\mathfrak{A}$.
\end{defn}
In the case that $\mathfrak{A}$ is a unital, simple, purely infinite C$^*$-algebra, $K_1(\mathfrak{A})$ is canonically isomorphic to $\mathfrak{A}^{-1}/\mathfrak{A}^{-1}_0$ by {\cite[Theorem 1.9]{Cu}}.  Thus if $N \in \mathfrak{A}$ is a normal operator such that $\Gamma(N)$ is trivial then $\lambda I_\mathfrak{A} - N \in \mathfrak{A}^{-1}_0$ for all $\lambda \notin \sigma(N)$.  Furthermore if $N \in \mathfrak{A}$ is a normal operator and $\lambda \notin \sigma(N)$ then $\Gamma(N)(\lambda)$ describes the connected component of $\lambda I_\mathfrak{A} - N$ in $\mathfrak{A}^{-1}$.
\par
The reason for examining the index function in the context of approximately unitarily equivalent normal operators is seen by the following necessary condition.
\begin{lem}
\label{ccoits}
Let $\mathfrak{A}$ be a unital and let $N_1,N_2 \in \mathfrak{A}$ be normal operators such that $N_1 \in \overline{\mathcal{S}(N_2)}$.  Then
\begin{enumerate}
	\item if $\lambda I_\mathfrak{A} - N_2 \in \mathfrak{A}^{-1}_0$ for some $\lambda \notin \sigma(N_1)$ then $\lambda I_\mathfrak{A} - N_1 \in \mathfrak{A}^{-1}_0$, and
	\item if $\mathfrak{A}$ is a unital, simple, purely infinite C$^*$-algebra then $\Gamma(N_1)(\lambda) = \Gamma(N_2)(\lambda)$ for all $\lambda \notin \sigma(N_1)$.
\end{enumerate}
\end{lem}
\begin{proof}
Suppose $N_1 \in \overline{\mathcal{S}(N_2)}$ and $\lambda \notin \sigma(N_1)$.  Then $\sigma(N_2) \subseteq \sigma(N_1)$ and there exists a sequence of invertible elements $V_n \in \mathfrak{A}$ such that 
\[
\lim_{n\to\infty} \left\|N_1 - V_nN_2V_n^{-1}\right\| = 0.
\]
Thus it is clear that
\[
\lim_{n\to\infty} \left\|(\lambda I_\mathfrak{A} - N_1) - V_n(\lambda I_\mathfrak{A} - N_2)V_n^{-1}\right\| = 0.
\]
Therefore, if $\lambda I_\mathfrak{A} - N_2 \in \mathfrak{A}^{-1}_0$ then $V_n(\lambda I_\mathfrak{A} - N_2)V_n^{-1} \in \mathfrak{A}^{-1}_0$ for all $n \in \mathbb{N}$ and thus first result trivially follows.
\par
In the case $\mathfrak{A}$ is a unital, simple, purely infinite C$^*$-algebra, the above implies that $\lambda I_\mathfrak{A} - N_1$ and $V_n(\lambda I_\mathfrak{A} - N_2)V_n^{-1}$ are in the same connected component of $\mathfrak{A}^{-1}$ for sufficiently large $n$.  Therefore
\[
\begin{array}{rcl}
[\lambda I_\mathfrak{A} - N_1]_1 &=&  [V_n(\lambda I_\mathfrak{A} - N_2)V_n^{-1}]_1 \\
 &=& [V_n]_1[\lambda I_\mathfrak{A} - N_2]_1[V_n^{-1}]_1 \\
 &=&  [\lambda I_\mathfrak{A} - N_2]_1.
 \end{array} 
\]
Hence $\Gamma(N_1)(\lambda) = \Gamma(N_2)(\lambda)$.
\end{proof}
The reason for the existence of an operator theoretic proof to {\cite[Theorem 1.7]{Dad}} is the K-theory of unital, simple, purely infinite C$^*$-algebras along with the following result due to Lin.
\begin{thm}[{\cite[Theorem 4.4]{Li}}]
\label{linfa}
Let $\mathfrak{A}$ be a unital, simple, purely infinite C$^*$-algebra and let $N \in \mathfrak{A}$ be a normal operator.  Then $N$ can be approximated by normal operators with finite spectra if and only if $\Gamma(N)$ is trivial. 
\end{thm}
Using Lin's result and the following trivial technical detail, we can easily provide an operator theoretic proof of {\cite[Theorem 1.7]{Dad}} for unital, simple, purely infinite C$^*$-algebras with trivial $K_0$-group and normal operators with trivial index function.
\begin{lem}
\label{densespectra2}
Let $\mathfrak{A}$ be a C$^*$-algebra, let $N \in \mathfrak{A}$ be a normal operator, let $U$ be an open subset of $\mathbb{C}$ such that $U \cap \sigma(N) \neq \emptyset$, and let $(N_n)_{n\geq1}$ be a sequence of normal operators from $\mathfrak{A}$ such that $N = \lim_{n\to\infty} N_n$.  Then there exists a $k \in \mathbb{N}$ such that $\sigma(N_n) \cap U \neq \emptyset$ for all $n \geq k$.
\end{lem}
\begin{proof}
The proof follows easily from the continuity of the Continuous Functional Calculus of Normal Operators.
\end{proof}
\begin{prop}
\label{aueuspik0}
Let $\mathfrak{A}$ be a unital, simple, purely infinite C$^*$-algebra such that $K_0(\mathfrak{A})$ is trivial.  Let $N_1, N_2 \in \mathfrak{A}$ be normal operators such that $\Gamma(N_1)$ and $\Gamma(N_2)$ are trivial.  Then $N_1 \sim_{au} N_2$ if and only if $\sigma(N_1) = \sigma(N_2)$.
\end{prop}
\begin{proof}
By previous discussions it is clear that $\sigma(N_1) = \sigma(N_2)$ if $N_1 \sim_{au} N_2$.  Suppose $\sigma(N_1) = \sigma(N_2)$.  Since $K_0(\mathfrak{A}) = \{0\}$, all non-trivial projections are Murray-von Neumann equivalent by {\cite[Theorem 1.4]{Cu}}.  Thus any two normal operators with the same finite spectrum are unitarily equivalent.
\par
By the assumption that $\Gamma(N_1)$ and $\Gamma(N_2)$ are trivial, $N_1$ and $N_2$ can be approximated by normal operators with finite spectrum by {\cite[Theorem 4.4]{Li}}.  By small perturbations using Lemma \ref{densespectra2} and the semicontinuity of the spectrum, we can assume that $N_1$ and $N_2$ can be approximated arbitrarily well by normal operators with the same finite spectrum.  Thus the result follows.
\end{proof}
Note the condition `$\Gamma(N_1)$ and $\Gamma(N_2)$ are trivial' holds when $\mathfrak{A}^{-1}_0 = \mathfrak{A}^{-1}$ or equivalently when $K_1(\mathfrak{A})$ is trivial (see {\cite[Theorem 1.9]{Cu}}).  
\par
If $\mathcal{O}_2$ is the Cuntz algebra generated by two isometries, $K_0(\mathcal{O}_2)$ and $K_1(\mathcal{O}_2)$ are trivial by {\cite[Theorem 3.7]{Cu}} and {\cite[Theorem 3.8]{Cu}} respectively.  Thus Proposition \ref{aueuspik0} completely classifies when two normal operators in $\mathcal{O}_2$ are approximately unitarily equivalent.
\begin{cor}
\label{o2results}
Let $N,M \in \mathcal{O}_2$ be normal operators.  Then $N \sim_{au} M$ if and only if $\sigma(N) = \sigma(M)$.
\end{cor}
Note that the proof of Proposition \ref{aueuspik0} is easily modified to a more general setting.  To see this, we recall the following definitions.
\begin{defn}
Let $\mathfrak{A}$ be a unital C$^*$-algebra.  We say that $\mathfrak{A}$ has the finite normal property (property (FN)) if every normal operator in $\mathfrak{A}$ is the limit of normal operators from $\mathfrak{A}$ with finite spectrum.  We say that $\mathfrak{A}$ has the weak finite normal property (property weak (FN)) if every normal operator $N \in \mathfrak{A}$ such that $\lambda I_\mathfrak{A} - N \in \mathfrak{A}^{-1}_0$ for all $\lambda \notin \sigma(N)$ is the limit of normal operators from $\mathfrak{A}$ with finite spectrum.
\end{defn}
\begin{cor}
\label{aueios}
Let $\mathfrak{A}$ be a unital C$^*$-algebra such that $\mathfrak{A}$ has property weak (FN) and any two non-zero projections in $\mathfrak{A}$ are Murray-von Neumann equivalent.  If $N_1,N_2 \in \mathfrak{A}$ are two normal operators such that $\lambda I_\mathfrak{A} - N_q \in \mathfrak{A}^{-1}_0$ for all $\lambda \notin \sigma(N_q)$ and $q \in \{1,2\}$ then $N_1 \sim_{au} N_2$ if and only if $\sigma(N_1) = \sigma(N_2)$.
\end{cor}
\begin{cor}
Let $\mathfrak{A}$ be a unital C$^*$-algebra such that $\mathfrak{A}$ has property (FN) and any two non-zero projections in $\mathfrak{A}$ are Murray-von Neumann equivalent.  If $N_1,N_2 \in \mathfrak{A}$ are two normal operators then $N_1 \sim_{au} N_2$ if and only if $\sigma(N_1) = \sigma(N_2)$.
\end{cor}
\begin{cor}
Let $\mathfrak{M}$ be a type III factor with separable predual and let $N_1, N_2 \in \mathfrak{M}$ be normal operators.  Then $N_1 \sim_{au} N_2$ if and only if $\sigma(N_1) = \sigma(N_2)$.
\end{cor}
Our next task is to provide an operator theoretic proof of {\cite[Theorem 1.7]{Dad}} when $K_0(\mathfrak{A})$ is non-trivial yet $K_1(\mathfrak{A})$ is trivial.  The Cuntz algebras, $\mathcal{O}_n$, generated by $n \in \mathbb{N} \cup \{\infty\}$ isometries (where $K_0(\mathcal{O}_n) = \mathbb{Z}_{n-1}$ and $K_1(\mathcal{O}_n)$ is trivial by {\cite[Theorem 3.7]{Cu}} and {\cite[Theorem 3.8]{Cu}} respectively) are excellent examples of such algebras.  We begin with the case that our two normal operators have the same connected spectrum.  The following lemma is motivated by the proof of {\cite[Theorem 2.8]{Sk}} and contains the essential ideas used in main result of this section (Theorem \ref{ausiuspi}) and in Section \ref{sec:DISTANCEUNITARYORBITS}.
\begin{lem}
\label{backandforth}
Let $\mathfrak{A}$ be a unital, simple, purely infinite C$^*$-algebra and let $N_1, N_2 \in \mathfrak{A}$ be normal operators.  Suppose that $\Gamma(N_1)$ and $\Gamma(N_2)$ are trivial, $\sigma(N_1) = \sigma(N_2)$, and $\sigma(N_1)$ is connected.  Then $N_1 \sim_{au} N_2$.
\end{lem}
\begin{proof}
We shall begin with the case that $\sigma(N_1) = \sigma(N_2) = [0,1]$ and then modify the proof for the general case.
\par
Suppose $\sigma(N_1) = [0,1] = \sigma(N_2)$.  Let $\epsilon > 0$ and choose $n \in \mathbb{N}$ such that $\frac{1}{n} < \epsilon$.  By {\cite[Theorem 4.4]{Li}} (or the fact that unital, simple, purely infinite C$^*$-algebras have real rank zero (see {\cite[Theorem V.7.4]{Da4}})), by Lemma \ref{densespectra2}, by the semicontinuity of the spectrum, and by perturbing eigenvalues, there exists two collections of non-zero, pairwise orthogonal projections 
\[
\left\{P^{(1)}_j\right\}^n_{j=0} \mbox{ and }\left\{P^{(2)}_j\right\}^n_{j=0}
\]
in $\mathfrak{A}$ such that
\[
\sum^n_{j=0} P^{(q)}_j = I_\mathfrak{A} \mbox{ and } \left\|N_q - \sum^n_{j=0} \frac{j}{n} P^{(q)}_j\right\| < 2\epsilon
\]
for all $q \in \{1,2\}$.  The idea of the proof is to apply a `back and forth' argument to produce a unitary that intertwines the approximations of $N_1$ and $N_2$.
\par
Since $\mathfrak{A}$ is a unital, simple, purely infinite C$^*$-algebra, $P^{(1)}_0$ is Murray-von Neumann equivalent to a proper subprojection of $P^{(2)}_0$.  Thus we can write $P^{(2)}_0 = Q^{(2)}_0 + R^{(2)}_0$ where $Q^{(2)}_0$ and $R^{(2)}_0$ are non-zero orthogonal projections in $\mathfrak{A}$ such that $Q^{(2)}_0$ and $P^{(1)}_0$ are Murray-von Neumann equivalent.  Furthermore $R^{(2)}_0$ is Murray-von Neumann equivalent to a proper subprojection of $P^{(1)}_1$.  Thus we can write $P^{(1)}_1 = Q^{(1)}_1 + R^{(1)}_1$ where $Q^{(1)}_1$ and $R^{(1)}_1$ are non-zero orthogonal projections in $\mathfrak{A}$ such that $Q^{(1)}_1$ and $R^{(2)}_0$ are Murray-von Neumann equivalent.
\par
For notional purposes, let $Q^{(1)}_0 := 0$, $R^{(1)}_0 := P^{(1)}_0$, $Q^{(2)}_n := P^{(2)}_n$, and $R^{(2)}_n := 0$.  By repeating this procedure (using $R^{(1)}_1$ in place of $P^{(1)}_0$), we obtain sets of non-zero, pairwise orthogonal projections 
\[
\left\{Q^{(1)}_j, R^{(1)}_j\right\}^n_{j=1} \mbox{ and } \left\{Q^{(2)}_j, R^{(2)}_j\right\}^{n-1}_{j=0}
\]
such that $P^{(q)}_j = Q^{(q)}_j + R^{(q)}_j$ for all $j \in \{0,\ldots, n\}$ and $q \in \{1,2\}$, $R^{(2)}_j$ is Murray-von Neumann equivalent to $Q^{(1)}_{j+1}$ for all $j \in \{0,\ldots, n-1\}$, and $R^{(1)}_j$ is Murray-von Neumann equivalent to $Q^{(2)}_j$ for all $j \in \{0,\ldots, n-1\}$.  Since
\[
I_\mathfrak{A} = \sum^n_{j=0} Q^{(1)}_j + R^{(1)}_j = \sum^{n}_{j=0} Q^{(2)}_j + R^{(2)}_j, \,\,\,\,\,\,\,\,(*)
\]
we note that
\[
\begin{array}{rcl}
\left[R^{(1)}_n\right]_0 &=& \left[I_\mathfrak{A}\right]_0 - \sum^n_{j=1}\left[Q^{(1)}_j\right]_0 - \sum^{n-1}_{j=0} \left[R^{(1)}_j\right]_0  \\
 &=& \left[I_\mathfrak{A}\right]_0 - \sum^n_{j=1}\left[R^{(2)}_{j-1}\right]_0 - \sum^{n-1}_{j=0} \left[Q^{(2)}_j\right]_0 \\
 &=&  \left[Q^{(2)}_n\right]_0.
 \end{array} 
\]
Hence $R^{(1)}_n$ and $Q^{(2)}_n$ are Murray-von Neumann equivalent by {\cite[Theorem 1.4]{Cu}}.
\par
Let $\{V_j\}^n_{j=0}\cup\{W_j\}^{n-1}_{j=0}$ be partial isometries in $\mathfrak{A}$ such that $V_j^*V_j = R^{(1)}_j$ and $V_jV_j^* = Q^{(2)}_j$ for all $j \in \{0,\ldots, n\}$, and $W_j^*W_j = Q^{(1)}_{j+1}$ and $W_jW_j^* = R^{(2)}_j$ for all $j \in \{0,\ldots, n-1\}$.  Hence $(*)$ implies that
\[
U := \sum^n_{j=0} V_j + \sum^{n-1}_{j=0} W_j 
\]
is a unitary operator in $\mathfrak{A}$.  Moreover
\[
\begin{array}{rcl}
 U^*\left(\sum^n_{j=0} \frac{j}{n} P^{(2)}_j\right)U &=& U^*\left(\sum^{n}_{j=0} \frac{j}{n}Q^{(2)}_j + \sum^{n}_{j=0} \frac{j}{n}R^{(2)}_j\right)U \\
 &=& \sum^{n}_{j=0} \frac{j}{n}R^{(1)}_j + \sum^{n-1}_{j=0} \frac{j}{n}Q^{(1)}_{j+1} .
 \end{array} 
\]
Hence, since
\[
\sum^n_{j=0} \frac{j}{n} P^{(1)}_j = \sum^n_{j=0} \frac{j}{n} Q^{(1)}_j + \sum^n_{j=0} \frac{j}{n} R^{(1)}_j,
\]
we obtain that
\[
\left\|N_1 - U^*N_2U\right\| \leq 5\epsilon.
\]
Since $\epsilon > 0$ was arbitrary, $N_1 \sim_{au} N_2$.
\par
To complete the general case, we will use a technique similar to that used in the proof of {\cite[Theorem 2.8]{Sk}}.  To begin, let $N_1$ and $N_2$ be as in the statement of the lemma.  Fix $\epsilon > 0$ and for each $(n,m) \in \mathbb{Z}^2$ let 
\[
B_{n,m} := \left(\epsilon n - \frac{\epsilon}{2}, \epsilon n + \frac{\epsilon}{2} \right]  + i \left( \epsilon m - \frac{\epsilon}{2}, \epsilon m + \frac{\epsilon}{2} \right] \subseteq \mathbb{C}.
\]
Thus the sets $B_{n,m}$ partition the complex plane into a grid with side-lengths $\epsilon$.
\par
For each $(n,m) \in \mathbb{Z}^2$ we label the box $B_{n,m}$ relevant if $\sigma(N_1) \cap B_{n,m} \neq \emptyset$ and we will say two boxes are adjacent if their union is connected.  Since $\sigma(N_1)$ is connected, the union of the relevant boxes is connected.
\par
By {\cite[Theorem 4.4]{Li}} we can approximate $N_1$ and $N_2$ within $\epsilon$ by normal operators $M_1$ and $M_2$ in $\mathfrak{A}$ with finite spectrum.  By Lemma \ref{densespectra2}, by the semicontinuity of the spectrum, and by perturbing eigenvalues, we can assume that $\sigma(M_q)$ is precisely the centres of the relevant boxes and $\left\|N_q - M_q\right\| \leq 2\epsilon$ for all $q \in \{1,2\}$.  
\par
We claim that there exists a unitary $U \in \mathfrak{A}$ such that $\left\|M_1 - U^*M_2U\right\| \leq \sqrt{2}\epsilon$.  Consider a tree $\mathcal{T}$ in $\mathbb{C}$ whose vertices are the centres of the relevant boxes and whose edges are straight lines that connect vertices in adjacent relevant boxes.  Consider a leaf of $\mathcal{T}$.  We can identify this leaf with the spectral projections of $M_1$ and $M_2$ corresponding to the eigenvalue defined by the vertex.  We can then apply the `back and forth' technique illustrated above to embed the spectral projection of $M_1$ under the corresponding spectral projection of $M_2$ and the remaining spectral projection of $M_2$ under a spectral projection of $M_1$ corresponding to the adjacent vertex of the leaf (which is within $\sqrt{2}\epsilon$).  By considering $\mathcal{T}$ with the above leaf removed, we then have a smaller tree.  By continually repeating this `back and forth'-crossing technique, we are eventually left with the trivial tree.  As before, $K$-theory implies the remaining projections are Murray-von Neumann equivalent.  It is then possible to use the partial isometries from the `back and forth' construction to create a unitary with the desired properties.
\end{proof}
Our next goal is to remove the condition `$\sigma(N_1)$ is connected' from Lemma \ref{backandforth}.  Unfortunately, two normal operators having equal spectrum is not enough to guarantee that the normal operators are approximately unitarily equivalent (even in the case that $K_1(\mathfrak{A})$ is trivial).  The technicality is the same as why two projections in $\mathcal{B}(\mathcal{H})$ are not always approximately unitarily equivalent.  To see this, we note the following lemmas.
\begin{lem}
\label{simorbitofprojections}
Let $\mathfrak{A}$ be a unital C$^*$-algebra and let $P,Q \in \mathfrak{A}$ be projections.  If there exists an element $V \in \mathfrak{A}^{-1}$ such that
\[
\left\|Q - VPV^{-1}\right\| < \frac{1}{2}
\]
then $P$ and $Q$ are Murray-von Neumann equivalent.
\end{lem}
\begin{proof}
Let $P_0 := VPV^{-1} \in \mathfrak{A}$ and let $Z := P_0Q + (I_\mathfrak{A} - P_0)(I_\mathfrak{A} - Q) \in \mathfrak{A}$.  Hence $P_0$ is an idempotent and it is clear that
\[
\begin{array}{rcl}
\left\|Z - I_\mathfrak{A}\right\| &=& \left\|(P_0Q + (I_\mathfrak{A} - P_0)(I_\mathfrak{A} - Q)) - (Q + (I_\mathfrak{A} - Q))\right\|  \\
 &\leq& \left\|(P_0 - I_\mathfrak{A})Q\right\| + \left\|((I_\mathfrak{A} - P_0) - I_\mathfrak{A})(I_\mathfrak{A} - Q)\right\| \\
 &=&  \left\|(P_0 - Q)Q\right\| + \left\|((I_\mathfrak{A} - P_0) - (I_\mathfrak{A} - Q))(I_\mathfrak{A} - Q)\right\| \\
 &\leq&  \left\|P_0 - Q\right\| + \left\|Q - P_0\right\| < 1.
 \end{array} 
\]
Hence $Z \in \mathfrak{A}^{-1}$.  Therefore, if $U$ is the partial isometry in the polar decomposition of $Z$, $Z = U|Z|$ and $U$ is a unitary element of $\mathfrak{A}$.
\par
We claim that $UQU^* = P_0$.  To see this, we notice that $U = Z|Z|^{-1}$,  $ZQ = P_0Q = P_0Z$, and
\[
Z^*Z = QP_0Q + (I_\mathfrak{A} - Q)(I_\mathfrak{A} - P_0)(I_\mathfrak{A} - Q).
\]
Thus $QZ^*Z = QP_0Q = Z^*ZQ$ so $Q$ commutes with $Z^*Z$.  Hence $Q$ commutes with $C^*(Z^*Z)$ and thus $Q$ commutes with $|Z|^{-1}$.  Thus
\[
\begin{array}{rcl}
UQU^* &=&  Z|Z|^{-1}Q|Z|^{-1}Z^* \\
 &=& ZQ|Z|^{-2}Z^* \\
 &=& P_0Z|Z|^{-2}Z^* = P_0
 \end{array} 
\]
as claimed.
\par
Therefore $Q = (U^*V)P(U^*V)^{-1}$ where $U^*V \in \mathfrak{A}^{-1}$.  It is standard to verify that if $W$ is the partial isometry in the polar decomposition of $U^*V$ then $W$ is a unitary such that $Q = WPW^*$ (see {\cite[Proposition 2.2.5]{RLL}}).  Therefore $P \sim_u Q$ and thus $P$ and $Q$ are Murray-von Neumann equivalent.
\end{proof}
\begin{lem}
\label{mvneopsaaueop}
Let $\mathfrak{A}$ be a unital, simple, purely infinite C$^*$-algebra and let $P$ and $Q$ be projections in $\mathfrak{A}$.  Then $P \sim_u Q$ if and only if $P \sim_{au} Q$ if and only if $Q \in \overline{\mathcal{S}(P)}$ only if $P$ and $Q$ are Murray-von Neumann equivalent.  If $P \neq I_\mathfrak{A}$ and $Q \neq I_\mathfrak{A}$, then $P \sim_u Q$ whenever $P$ and $Q$ are Murray-von Neumann equivalent.
\end{lem}
\begin{proof}
The result follows from Lemma \ref{simorbitofprojections} and standard K-theory arguments.
\end{proof}
The above shows that if $\mathfrak{A}$ is a unital, simple, purely infinite C$^*$-algebra with $K_0(\mathfrak{A})$ being non-trivial, there exists two projections $P,Q \in \mathfrak{A}$ with $\sigma(P) = \sigma(Q) = \{0,1\}$ that are not approximately unitarily equivalent.  Thus knowledge of the spectrum is not enough to complete our classification. 
\par
To avoid the above technicality, we will describe an additional condition for two normal operators to be approximately unitarily equivalent in a unital C$^*$-algebra.  The construction of this conditions makes use of the analytical functional calculus.
\begin{lem}
\label{analytic}
Let $\mathfrak{A}$ be a unital C$^*$-algebra, let $A, B \in \mathfrak{A}$, and let $f : \mathbb{C} \to \mathbb{C}$ be a function that is analytic on an open neighbourhood $U$ of $\sigma(A)\cup \sigma(B)$.  If $A \in \overline{\mathcal{S}(B)}$ then $f(A) \in \overline{\mathcal{S}(f(B))}$.  Similarly if $A \sim_{au} B$ then $f(A) \sim_{au} f(B)$.
\end{lem}
\begin{proof}
Let $(V_n)_{n\geq 1}$ be a sequence of invertible elements in $\mathfrak{A}$ such that 
\[
\lim_{n\to\infty} \left\|A - V_nBV_n^{-1}\right\| = 0.
\]
Let $\gamma$ be any compact, rectifiable curve inside $U$ such that $(\sigma(A)\cup \sigma(B)) \cap \gamma = \emptyset$, $Ind_\gamma(z) \in \{0,1\}$ for all $z \in \mathbb{C} \setminus \gamma$, $Ind_\gamma(z) = 1$ for all $z \in \sigma(A)\cup \sigma(B)$, and $\{z \in \mathbb{C} \, \mid \, Ind_\gamma(z) \neq 0\} \subseteq U$.  Then
\[
\begin{array}{rcl}
 &&  f(A) - V_nf(B)V_n^{-1} \\
 &=&  \frac{1}{2\pi i} \int_\gamma f(z)\left( (z I_\mathfrak{A} - A)^{-1} - V_n(z I_\mathfrak{A} - B)^{-1}V_n^{-1} \right) dz \\
 &=& \frac{1}{2\pi i} \int_\gamma f(z)\left( (z I_\mathfrak{A} - A)^{-1} - (z I_\mathfrak{A} - V_nBV_n^{-1})^{-1} \right) dz \\
 &=& \frac{1}{2\pi i} \int_\gamma f(z)(z I_\mathfrak{A} - A)^{-1}(A - V_nBV_n^{-1})(z I_\mathfrak{A} - V_nBV_n^{-1})^{-1} dz.
 \end{array} 
\]
Hence $\left\|f(A) - V_nf(B)V_n^{-1}\right\|$ is at most
\[
\frac{length(\gamma) \left\|A - V_nBV_n^{-1}\right\|}{2\pi}   \sup_{z \in \gamma} |f(z)|\left\|(z I_\mathfrak{A} - A)^{-1}\right\| \left\|(z I_\mathfrak{A} - V_nBV_n^{-1})^{-1}\right\|.
\]
Provided $\left\|A - V_nBV_n^{-1}\right\| \left\|(z I_\mathfrak{A} - A)^{-1}\right\| < 1$ for all $z \in \gamma$, the second resolvent equation can be used to show that
\[
\left\|(z I_\mathfrak{A} - V_nBV_n^{-1})^{-1}\right\| \leq \frac{\left\|(z I_\mathfrak{A} - A)^{-1}\right\|}{1 - \left\|A - V_nBV_n^{-1}\right\| \left\|(z I_\mathfrak{A} - A)^{-1}\right\|}
\]
for all $z \in \gamma$.  Since $\lim_{n\to\infty} \left\|A - V_nBV_n^{-1}\right\| = 0$, $\gamma$ is compact, and the resolvent function of an operator is continuous on the resolvent, $\left\|f(A) - V_nf(B)V_n^{-1}\right\|$ is at most
\[
\frac{length(\gamma) \left\|A - V_nBV_n^{-1}\right\|}{2\pi}   \sup_{z \in \gamma} |f(z)|\frac{\left\|(z I_\mathfrak{A} - A)^{-1}\right\|^2}{1 - \left\|A - V_nBV_n^{-1}\right\| \left\|(z I_\mathfrak{A} - A)^{-1}\right\|}
\]
for sufficiently large $n$.  Since the resolvent function is a continuous function on the resolvent of an operator and $\gamma$ is compact, the above supremum is finite and tends to
\[
\sup_{z \in \gamma} |f(z)|\left\|(z I_\mathfrak{A} - A)^{-1}\right\|^2
\]
as $n \to \infty$.  Thus, as
\[
\lim_{n\to\infty} \left\|A - V_nBV_n^{-1}\right\| = 0
\]
and $length(\gamma)$ is finite, $f(A) \in \overline{\mathcal{S}(f(B))}$.
\par
The proof that $A \sim_{au} B$ implies $f(A) \sim_{au} f(B)$ follows directly by replacing the invertible elements $V_n$ with unitary operators.
\end{proof}
If $\mathfrak{A}$ in Lemma \ref{analytic} were a unital, simple, purely infinite C$^*$-algebra, if $A$ and $B$ were normal operators, and if $f$ took values in $\{0,1\}$ with $f(A)$ and $f(B)$ being non-trivial, then Lemma \ref{mvneopsaaueop} would imply that the projections $f(A)$ and $f(B)$ are Murray-von Neumann equivalent in $\mathfrak{A}$.  Thus, to simplify notation, we make the following definition.
\begin{defn}
Let $\mathfrak{A}$ be a unital C$^*$-algebra and let $N_1, N_2 \in \mathfrak{A}$ be normal operators.  We say that $N_1$ and $N_2$ have equivalent common spectral projections if for every function $f : \mathbb{C} \to \mathbb{C}$ that is analytic on an open neighbourhood $U$ of $\sigma(N_1) \cup \sigma(N_2)$ with $f(U) \subseteq \{0,1\}$, the projections $f(N_1)$ and $f(N_2)$ are Murray-von Neumann equivalent.
\end{defn}
If $\mathfrak{A}$ is a unital, simple, purely infinite C$^*$-algebra and $\sigma(N_1) = \sigma(N_2)$, it is elementary to show that using {\cite[Theorem 1.4]{Cu}} that $N_1$ and $N_2$ have equivalent spectral projections if and only if they induce the same group homomorphisms from $K_0(\sigma(N_1))$ to $K_0(\mathfrak{A})$ via the Continuous Functional Calculus of Normal Operators.
\par
Finally, with the above and the arguments used in Lemma \ref{backandforth}, we provide an operator theoretic proof of {\cite[Theorem 1.7]{Dad}} for planar compact sets in the case that $K_1(\mathfrak{A})$ is trivial.
\begin{thm}
\label{ausiuspi}
Let $\mathfrak{A}$ be a unital, simple, purely infinite C$^*$-algebra and let $N_1, N_2 \in \mathfrak{A}$ be normal operators.  Suppose
\begin{enumerate}
	\item $\sigma(N_1) = \sigma(N_2)$,
	\item $\Gamma(N_1)$ and $\Gamma(N_2)$ are trivial, and
	\item $N_1$ and $N_2$ have equivalent common spectral projections.
\end{enumerate}
Then $N_1 \sim_{au} N_2$.
\end{thm}
\begin{proof}
Fix $\epsilon > 0$ and consider the $\epsilon$-grid used in Lemma \ref{backandforth}.  We label the box $B_{n,m}$ relevant if $B_{n,m} \cap \sigma(N_1) \neq \emptyset$.  Let $K$ be the union of the relevant boxes.  Since $\sigma(N_1)$ is compact, $K$ has finitely many connected components.  Let $L_1,\ldots, L_k$ be the connected components of $K$.  By construction $dist(L_i, L_j) \geq \epsilon$ for all $i \neq j$.  Therefore, if $f_i$ is the characteristic function of $L_i$, the third assumptions of the theorem implies $f_i(N_1)$ and $f_i(N_2)$ are Murray-von Neumann equivalent for each $i \in \{1,\ldots, k\}$.
\par
Note the second assumption of the theorem implies that there exists normal operators $M_1$ and $M_2$ in $\mathfrak{A}$ with finite spectrum such that $\left\|N_q - M_q\right\| < \epsilon$ for all $q \in \{1,2\}$.  By an application of Lemma \ref{densespectra2}, by the semicontinuity of the spectrum, and by small perturbations, we can assume that $M_q$ has spectrum contained in $K$ and $\sigma(M_q) \cap B_{n,m} \neq \emptyset$ for all relevant boxes $B_{n,m}$ and $q \in \{1,2\}$.  Furthermore, since each $f_i$ extends to a continuous function on an open neighbourhood of $K$, we can assume that $\left\|f_i(N_q) - f_i(M_q)\right\| < \frac{1}{2}$ for all $i \in \{1,\ldots, k\}$ and $q \in \{1,2\}$ by properties of the continuous functional calculus.  Therefore, for each $i \in \{1,\ldots, k\}$ and $q \in \{1,2\}$, $f_i(N_q)$ and $f_i(M_q)$ can be assumed to be Murray-von Neumann equivalent by Lemma \ref{simorbitofprojections}.  Since $f_i(N_1)$ and $f_i(N_2)$ are Murray-von Neumann equivalent for each $i \in \{1,\ldots, k\}$, $f_i(M_1)$ and $f_i(M_2)$ are Murray-von Neumann equivalent for each $i \in \{1,\ldots, k\}$.  By perturbing the spectrum of $M_1$ and $M_2$ inside each $L_i$, we can assume that $\sigma(M_q)$ is precisely the centres of the relevant boxes for all $q \in \{1,2\}$, $f_i(M_1)$ and $f_i(M_2)$ are Murray-von Neumann equivalent for each $i \in \{1,\ldots, k\}$, and $\left\|N_q - M_q\right\| < 2\epsilon$ for all $q \in \{1,2\}$.
\par
Next we apply the `back and forth' argument of Lemma \ref{backandforth} to the spectrum of $M_1$ and $M_2$ in each $L_i$ separately.  This process can be applied to each $L_i$ separately as in Lemma \ref{backandforth} due to the fact that $f_i(M_1)$ and $f_i(M_2)$ are Murray-von Neumann equivalent so the final step of the construction (that is, $R^{(1)}_n$ and $Q^{(2)}_n$ are Murray-von Neumann equivalent) can be completed.  Thus, for each $i \in \{1, \ldots, k\}$, the `back and forth' process produces a partial isometry $V_i \in \mathfrak{A}$ such that $V_i^*V_i = f_i(M_1)$, $V_iV_i^* = f_i(M_2)$, and $\left\|M_1 f_i(M_1) - V_i^* M_2 f_i(M_2)V_i\right\| \leq \sqrt{2}\epsilon$.  Therefore, if $U := \sum^k_{i=1} V_i$ then $U \in \mathfrak{A}$ is a unitary as
\[
\sum^k_{i=1} f_i(M_1) = I_\mathfrak{A} = \sum^k_{i=1} f_i(M_2)
\]
are sums of orthogonal projections.  Moreover, a trivial computation shows
\[
\left\|M_1 - U^*M_2U\right\| \leq \sqrt{2}\epsilon
\]
so
\[
\left\|N_1 - U^*N_2U \right\| \leq (4 + \sqrt{2})\epsilon
\]
completing the proof.
\end{proof}
\begin{cor}
\label{aueintrivialk1}
Let $\mathfrak{A}$ be a unital, simple, purely infinite C$^*$-algebra such that $K_1(\mathfrak{A})$ is trivial and let $N_1, N_2 \in \mathfrak{A}$ be normal operators.  Then $N_1 \sim_{au} N_2$ if and only if
\begin{enumerate}
	\item $\sigma(N_1) = \sigma(N_2)$ and
	\item $N_1$ and $N_2$ have equivalent common spectral projections.
\end{enumerate}
\end{cor}
\begin{proof}
One direction is follows from Theorem \ref{ausiuspi} and the fact that $K_1(\mathfrak{A})$ is trivial implies $\mathfrak{A}^{-1} = \mathfrak{A}_0^{-1}$ by {\cite[Theorem 1.9]{Cu}}.  The other direction follows from Lemma \ref{analytic} and Lemma \ref{mvneopsaaueop}.
\end{proof}

\section{Distance Between Unitary Orbits of Normal Operators}
\label{sec:DISTANCEUNITARYORBITS}

In this section we will make use of the techniques of Section \ref{sec:CLOSEDUNITARYORBITS} to provide some bounds for the distance between the unitary orbits of two normal operator in unital, simple, purely infinite C$^*$-algebras.  In particular, Corollary \ref{hausdorffdistanceforsimilarity} can be used to deduce Theorem \ref{ausiuspi}.  Furthermore, these results along with {\cite[Theorem 1.7]{Dad}} will provide information about the distance between unitary orbits of normal operators with non-trivial index function.
\par
We begin with the following definition that is common in the discussion of the distance between unitary orbits.
\begin{defn}
Let $X$ and $Y$ be subsets of $\mathbb{C}$.  The Hausdorff distance between $X$ and $Y$, denoted $d_H(X, Y)$, is
\[
d_H(X, Y) := \max\left\{ \sup_{x \in X} dist(x, Y), \sup_{y \in Y} dist(y, X)\right\}.
\]
\end{defn}
In \cite{Da1}, Davidson developed the following notation for the Calkin algebra that will be of particular use to us.
\begin{defn}
\label{rho}
Let $\mathfrak{A}$ be a unital, simple, purely infinite C$^*$-algebra.  For normal operators $N_1, N_2 \in \mathfrak{A}$ let $\rho(N_1, N_2)$ denote the maximum of $d_H(\sigma(N_1), \sigma(N_2))$ and
\[
\sup\{ dist(\lambda, \sigma(N_1)) + dist(\lambda, \sigma(N_2)) \, \mid \, \lambda \notin \sigma(N_1) \cup \sigma(N_2), \Gamma(N_1)(\lambda) \neq \Gamma(N_2)(\lambda)\}.
\]
\end{defn}
We begin by noting the following adaptation of {\cite[Proposition 1.2]{Da1}}.
\begin{prop}
\label{lowerbounddistantofunitaryorbits}
Let $\mathfrak{A}$ be a unital C$^*$-algebra and let $N_1, N_2 \in \mathfrak{A}$ be normal operators.  Then
\[
dist(\mathcal{U}(N_1), \mathcal{U}(N_2)) \geq d_H(\sigma(N_1), \sigma(N_2)).
\]
If $\mathfrak{A}$ is a unital, simple, purely infinite C$^*$-algebra then
\[
dist(\mathcal{U}(N_1), \mathcal{U}(N_2)) \geq \rho(N_1, N_2).
\]
\end{prop}
\begin{proof}
The proof of the first statement follows from {\cite[Proposition 2.1]{Da2}} and the proof of the second statement follows from the proof of {\cite[Proposition 1.2]{Da1}} where the index function $\Gamma$ is substituted for the traditional index function.
\end{proof}
For our discussions of the distance between unitary orbits of normal operators in unital, simple, purely infinite C$^*$-algebras, we shall begin with the case our normal operators have trivial index function so that $\rho(N_1, N_2) = d_H(\sigma(N_1), \sigma(N_2))$ and we may apply the techniques from Section \ref{sec:CLOSEDUNITARYORBITS}.
\par
We first turn our attention to the Cuntz algebra $\mathcal{O}_2$.  As $K_0(\mathcal{O}_2)$ and $K_1(\mathcal{O}_2)$ are trivial, we are led to the following generalization of {\cite[Theorem 1.5]{HN}} whose proof is identical to the one given below.
\begin{prop}[see {\cite[Theorem 1.5]{HN}}]
\label{hausdorffdist}
Let $\mathfrak{A}$ be a unital C$^*$-algebra such that $\mathfrak{A}$ has property weak (FN), any two non-zero projections in $\mathfrak{A}$ are Murray-von Neumann equivalent, and every non-zero projection in $\mathfrak{A}$ is properly infinite.  Let $N_1, N_2 \in \mathfrak{A}$ be normal operators such that $\Gamma(N_1)$ and $\Gamma(N_2)$ are trivial.  Then
\[
dist(\mathcal{U}(N_1), \mathcal{U}(N_2)) = d_H(\sigma(N_1), \sigma(N_2)).
\]
\end{prop}
\begin{proof}
One inequality follows from Proposition \ref{lowerbounddistantofunitaryorbits}.  Let $\epsilon > 0$.  Since $\mathfrak{A}$ has weak (FN), the conditions on $N_1$ and $N_2$ imply that there exists two normal operators $M_1,M_2\in\mathfrak{A}$ with finite spectrum such that $\left\|N_q - M_q\right\| < \epsilon$ for all $q \in \{1,2\}$.  By Lemma \ref{densespectra2}, by the semicontinuity of the spectrum, and by applying small perturbations, we may assume that $\sigma(M_q) \subseteq \sigma(N_q)$ and $\sigma(M_q)$ is an $\epsilon$-net for $\sigma(N_q)$ for all $q \in \{1,2\}$.
\par
Let $X$ be the set of all ordered pairs $(\lambda, \mu) \in \sigma(M_1) \times \sigma(M_2)$ such that either
\[
|\lambda - \mu| = dist(\lambda, \sigma(M_2)) \mbox{ or }|\lambda - \mu| = dist(\mu, \sigma(M_1)).
\]
For each $\lambda\in\sigma(M_1)$ and $\mu\in\sigma(M_2)$, let $n_\lambda := |\{(\lambda, \zeta) \in X\}|$ and $m_\mu := |\{(\zeta, \mu) \in X\}|$.  Clearly $n_\lambda \geq 1$ for all $\lambda\in\sigma(M_1)$, $m_\mu \geq 1$ for all $\mu\in\sigma(M_2)$, and $\sum_{\lambda \in \sigma(M_1)} n_\lambda = \sum_{\mu \in \sigma(M_2)} m_\mu$.
\par
Since every projection in $\mathfrak{A}$ is properly infinite, we can write 
\[
M_1 = \sum_{\lambda\in \sigma(N_0)} \sum^{n_\lambda}_{k=1} \lambda P_{\lambda, k}\,\,\,\,\,\,\,\mbox{ and }\,\,\,\,\,\,\, M_2 = \sum_{\mu\in \sigma(M_0)} \sum^{m_\mu}_{k=1} \mu Q_{\mu, k}
\]
where $\left\{\{P_{\lambda, k}\right\}^{n_\lambda}_{k=1}\}_{\lambda \in \sigma(M_1)}$ and $\left\{\{Q_{\mu, k}\}^{m_\mu}_{k=1}\right\}_{\mu \in \sigma(M_2)}$ are sets of non-zero orthogonal projections in $\mathfrak{A}$ each of which sums to the identity.  Since all projections in $\mathfrak{A}$ are Murray-von Neumann equivalent, using $X$ we can pair off the projections in these finite sums to obtain a unitary $U \in \mathfrak{A}$ (that is a sum of partial isometries) such that 
\[
\left\|M_1 - UM_2U^*\right\| \leq \sup\{|\lambda - \mu|\, \mid \, (\lambda, \mu) \in X\} = d_H(\sigma(M_1), \sigma(M_2)).
\]
Hence
\[
dist(\mathcal{U}(N_1), \mathcal{U}(N_2)) \leq 2\epsilon + d_H(\sigma(M_1), \sigma(M_2)).
\]
Since $\sigma(M_1)$ is an $\epsilon$-net for $\sigma(N_1)$, and $\sigma(M_2)$ is an $\epsilon$-net for $\sigma(N_2)$, 
\[
d_H(\sigma(M_1), \sigma(M_2)) \leq d_H(\sigma(N_1), \sigma(N_2)) + \epsilon
\]
completing the proof.
\end{proof}
Unfortunately Proposition \ref{hausdorffdist} does not completely generalize to unital, simple, purely infinite C$^*$-algebras with non-trivial $K_0$-group.  The following uses the ideas of Section \ref{sec:CLOSEDUNITARYORBITS} to obtain a preliminary result.
\begin{lem}
\label{hausdistconnspec}
Let $\mathfrak{A}$ be a unital, simple, purely infinite C$^*$-algebra and let $N_1, N_2 \in \mathfrak{A}$ be normal operators such that $\Gamma(N_1)$ and $\Gamma(N_2)$ are trivial.  If $\sigma(N_1)$ is connected then
\[
dist(\mathcal{U}(N_1), \mathcal{U}(N_2)) = d_H(\sigma(N_1), \sigma(N_2)).
\]
\end{lem}
\begin{proof}
One inequality follows from Proposition \ref{lowerbounddistantofunitaryorbits}.  The proof of the other inequality is a more complicated `back and forth' argument.  Fix $\epsilon > 0$ and let $B_{n,m}$ be as in Lemma \ref{backandforth}.  For each $q \in \{1,2\}$, we will say that $B_{n,m}$ is $N_q$-relevant if $B_{n,m} \cap \sigma(N_q) \neq \emptyset$.  By {\cite[Theorem 4.4]{Li}} there exists normal operators $M_1, M_2 \in \mathfrak{A}$ with finite spectrum such that $\left\|N_q - M_q\right\| < \epsilon$ for all $q = \{1,2\}$.  By Lemma \ref{densespectra2}, by the semicontinuity of the spectrum, and by a small perturbation, we can assume that $\sigma(M_q)$ is precisely the centres of the $N_q$-relevant boxes and $\left\|N_q - M_q\right\| \leq 2\epsilon$.  For each $q \in \{1,2\}$ and $\lambda \in \sigma(M_q)$ let $P^{(q)}_\lambda$ be the non-zero spectral projection of $M_q$ corresponding to $\lambda$.  
\par
To begin our `back and forth' argument, we will construct a bipartite graph, $\mathcal{G}$, using $\sigma(M_1)$ and $\sigma(M_2)$ as vertices (where we have two vertices for $\lambda$ if $\lambda \in \sigma(M_1) \cap \sigma(M_2)$).  The process for constructing the edges in $\mathcal{G}$ is as follows: for each $i,j \in \{1,2\}$ with $i \neq j$ and each $\lambda \in \sigma(M_i)$, for every $\mu \in \sigma(M_j)$ such that
\[
|\lambda - \mu| \leq 2\sqrt{2}\epsilon + d_H(\sigma(N_1), \sigma(N_2))
\]
(note that at least one such $\mu$ exists) add edges to $\mathcal{G}$ from $\mu$ to $\lambda$ and the centre of any $N_i$-relevant box adjacent (including diagonally adjacent) to the $N_i$-relevant box $\lambda$ describes.
\par
Clearly $\mathcal{G}$ is a bipartite graph and, by construction, if $\lambda \in \sigma(M_1)$ and $\mu \in \sigma(M_2)$ are connected by an edge of $\mathcal{G}$ then $|\lambda - \mu| \leq 2\sqrt{2}\epsilon + d_H(\sigma(N_1), \sigma(N_2))$.  We claim that $\mathcal{G}$ is connected.  To see this, we note that since $\mathcal{G}$ is bipartite and every vertex is the endpoint of at least one edge, it suffices to show that for each pair $\lambda, \mu \in \sigma(M_1)$ there exists a path from $\lambda$ to $\mu$.  Fix a pair $\lambda, \mu \in \sigma(M_1)$.  Since $\sigma(N_1)$ is connected, the union of the $N_1$-relevant boxes is connected so there exists a finite sequence $\lambda = \lambda_0, \lambda_1, \ldots, \lambda_k = \mu$ where $\lambda_{\ell-1}$ and $\lambda_{\ell}$ are centres of adjacent $N_1$-relevant boxes for all $\ell \in \{1,\ldots, k\}$.  However $\lambda_{\ell-1}$ and $\lambda_\ell$ are connected in $\mathcal{G}$ (via an element of $\sigma(M_2)$) by construction.  Hence the claim follows.
\par
Now that $\mathcal{G}$ is constructed, we will progressively remove vertices and edges from $\mathcal{G}$ and modify the non-zero projections $\left\{\left\{P^{(q)}_\lambda\right\}_{\lambda \in \sigma(M_j)}\right\}_{q\in \{1,2\}}$ in a specific manner to construct partial isometries in $\mathfrak{A}$ that will enable us to create a unitary $U \in \mathfrak{A}$ such that
\[
\left\|M_1 - U^*M_2U\right\| \leq 2\sqrt{2}\epsilon + d_H(\sigma(N_1), \sigma(N_2)).
\]
Since $\mathcal{G}$ is a connected graph, there exists a $j \in \{1,2\}$ and a vertex $\lambda \in \sigma(M_j)$ in $\mathcal{G}$ whose removal (along with all edges with $\lambda$ as an endpoint) does not disconnect $\mathcal{G}$.  Choose any vertex $\mu$ in $\mathcal{G}$ connected to $\lambda$ by an edge.  By the construction of $\mathcal{G}$ $|\lambda - \mu| \leq 2\sqrt{2}\epsilon + d_H(\sigma(N_1), \sigma(N_2))$ and $\mu \in \sigma(M_i)$ where $i \in \{1,2\}\setminus \{j\}$.  Since $\mathfrak{A}$ is a unital, simple, purely infinite C$^*$-algebra and $P^{(i)}_\mu$ is non-zero, there exists non-zero projections $Q^{(i)}_\mu$ and $R^{(i)}_\mu$ in $\mathfrak{A}$ such that $P^{(j)}_\lambda$ and $Q^{(i)}_\mu$ are Murray-von Neumann equivalent and $P^{(i)}_\mu = Q^{(i)}_\mu + R^{(i)}_\mu$.   To complete our recursive step, remove $\lambda$ from $\mathcal{G}$ (so $\mathcal{G}$ will still be a connected, bipartite graph), remove $P^{(j)}_\lambda$ from our list of projections, and replace $P^{(i)}_\mu$ with $R^{(i)}_\mu$ in our list of projections. 
\par
Continue the recursive process in the above paragraph until two vertices are left in $\mathcal{G}$ that must be connected by an edge.  Since $\mathcal{G}$ is bipartite, one of these two remaining vertices is a non-zero subprojection of a spectral projection of $M_1$ and the other is a non-zero subprojection of a spectral projection of $M_2$.  These two projections are Murray-von Neumann equivalent by the same $K$-theory argument used in Lemma \ref{backandforth}.
\par
By the same arguments as Lemma \ref{backandforth}, the Murray-von Neumann equivalence of the projections created in the above process allows us to create partial isometries and thus, by taking a sum, a unitary $U \in \mathfrak{A}$ with the claimed property.  Hence
\[
\left\|N_1 - U^*N_2U\right\| \leq (4 + 2\sqrt{2})\epsilon + d_H(\sigma(N_1), \sigma(N_2)).
\]
As $\epsilon > 0$ was arbitrary, the result follows.
\end{proof}
The above proof can be modified to show the following results.
\begin{cor}
Let $\mathfrak{A}$ be a unital, simple, purely infinite C$^*$-algebra and let $N_1, N_2 \in \mathfrak{A}$ be normal operators such that $\Gamma(N_1)$ and $\Gamma(N_2)$ are trivial.  Suppose for each $q \in \{1,2\}$ that $\sigma(N_q) = \bigcup^n_{i=1} K^{(q)}_i$ is a disjoint union of compact sets with $K^{(1)}_i$ connected for all $i \in \{1,\ldots, n\}$.  Let $\chi^{(q)}_i$ be the characteristic function of $K^{(q)}_i$ for all $q \in \{1,2\}$ and $i \in \{1,\ldots, n\}$.  If $\chi^{(1)}_i(N_1)$ and $\chi^{(2)}_i(N_2)$ are Murray-von Neumann equivalent for all $i \in \{1,\ldots, n\}$ then
\[
dist(\mathcal{U}(N_1), \mathcal{U}(N_2)) \leq \max_{i \in \{1,\ldots, n\}} d_H\left(K^{(1)}_i, K^{(2)}_i\right).
\]
\end{cor}
\begin{proof}
Fix $\epsilon > 0$.  The condition that `$\chi^{(1)}_i(N_1)$ and $\chi^{(2)}_i(N_2)$ are Murray-von Neumann equivalent' allows the arguments of Lemma \ref{hausdistconnspec} to be applied on each pair $\left(K^{(1)}_i, K^{(2)}_i\right)$ to produce a partial isometry $V_i \in \mathfrak{A}$ such that $V_i^*V_i = \chi^{(1)}_i(N_1)$, $V_iV_i^* = \chi^{(2)}_i(N_2)$, and 
\[
\left\| N_1\chi^{(1)}_i(N_1)  - V_i^* N_2\chi^{(2)}_i(N_2)V_i\right\| < \epsilon +  d_H\left(K^{(1)}_i, K^{(2)}_i\right).
\]
If $U := \sum^k_{i=1} V_i \in \mathfrak{A}$ then $U$ is a unitary operator such that 
\[
\left\|N_1 - U^*N_2U\right\| < \epsilon + \max_{i \in \{1,\ldots, n\}} d_H\left(K^{(1)}_i, K^{(2)}_i\right).
\]
Hence the result follows.
\end{proof}
\begin{lem}
\label{unionconnected}
Let $\mathfrak{A}$ be a unital, simple, purely infinite C$^*$-algebra and let $N_1, N_2 \in \mathfrak{A}$ be normal operators such that $\Gamma(N_1)$ and $\Gamma(N_2)$ are trivial.  Suppose $\sigma(N_1) \cup \sigma(N_2)$ is connected.  Then
\[
dist(\mathcal{U}(N_1), \mathcal{U}(N_2)) = d_H(\sigma(N_1), \sigma(N_2)).
\]
\end{lem}
\begin{proof}
The only caveat in the proof of Lemma \ref{hausdistconnspec} is that we require that the bipartite graph $\mathcal{G}$ is connected.  For each $q \in \{1,2\}$ let $K^{(q)}$ be the union of the $N_q$-relevant $B_{n,m}$.  Therefore, each $K^{(q)}$ is a union of finitely many connected components.  For each $q \in \{1,2\}$ let $\left\{K^{(q)}_k\right\}^{n_q}_{k=1}$ be the connected components of $K^{(q)}$.  By the construction of $\mathcal{G}$ in Lemma \ref{hausdistconnspec}, for each $q \in \{1,2\}$ all vertices from $\sigma(M_q)$ inside $K^{(q)}_k$ are connected in $\mathcal{G}$.  Moreover, for $i,j \in \{1,2\}$ with $i\neq j$, each vertex from $\sigma(M_j)$ inside $K^{(i)}_k$ is connected to each vertex of $K^{(j)}_\ell$ provided that $K^{(i)}_k \cup K^{(j)}_\ell$ is connected.  Since $\sigma(N_1) \cup \sigma(N_2)$ is connected, 
\[
\left(\bigcup^{n_1}_{k=1} K^{(1)}_k\right) \cup\left(\bigcup^{n_2}_{k=1} K^{(2)}_k\right)
\]
is connected and thus $\mathcal{G}$ is connected.  The remainder of the proof then follows as in Lemma \ref{hausdistconnspec}.
\end{proof}
Combining the ideas of the above results, we obtain the following
\begin{prop}
\label{unionwithinepsilondistanceresult}
Let $\mathfrak{A}$ be a unital, simple, purely infinite C$^*$-algebra, let $\epsilon > 0$, and let $N_1, N_2 \in \mathfrak{A}$ be normal operators such that $\Gamma(N_1)$ and $\Gamma(N_2)$ are trivial.  Let $B_{n,m}$ be as in Lemma \ref{backandforth}.  For each $q \in \{1,2\}$, we will say that $B_{n,m}$ is $N_q$-relevant if $B_{n,m} \cap \sigma(N_q) \neq \emptyset$.  For each $q \in \{1,2\}$ let $K^{(q)}$ be the union of the $N_q$-relevant boxes.  Suppose there exists an $n \in \mathbb{N}$ such that for each $q \in \{1,2\}$ we can write $K^{(q)} = \bigcup^n_{i=1} K^{(q)}_i$ where each $K^{(q)}_i$ is the union of finitely many connected sets, each $K^{(1)}_i \cup K^{(2)}_i$ is connected, and
\[
dist\left(K^{(1)}_i \cup K^{(2)}_i, K^{(1)}_j \cup K^{(2)}_j\right) > 0
\]
whenever $i \neq j$.  Let $\chi^{(q)}_i$ be the characteristic function of $K^{(q)}_i$ for all $q \in \{1,2\}$ and $i \in \{1,\ldots, n\}$.  If $\chi^{(1)}_i(N_1)$ and $\chi^{(2)}_i(N_2)$ are Murray-von Neumann equivalent for all $i \in \{1,\ldots, n\}$ then
\[
dist(\mathcal{U}(N_1), \mathcal{U}(N_2)) \leq (4 + 2\sqrt{2})\epsilon + \max_{i \in \{1,\ldots, n\}} d_H\left(K^{(1)}_i, K^{(2)}_i\right).
\]
\end{prop}
\begin{proof}
Let $\mathcal{G}$ be the bipartite graph described in Lemma \ref{hausdistconnspec} for this selection of $\epsilon$.  The graph $\mathcal{G}$ is not connected but the conditions of this proposition allows the proof of Lemma \ref{unionconnected} to performed on the vertices of $K^{(1)}_i \cup K^{(2)}_i$ separately to construct partial isometries $V_i \in \mathfrak{A}$ such that $V_i^*V_i = \chi^{(1)}_i(N_1)$, $V_iV_i^* = \chi^{(2)}_i(N_2)$, and 
\[
\left\| N_1\chi^{(1)}_i(N_1)  - V_i^* N_2\chi^{(2)}_i(N_2)V_i\right\| \leq (4 + 2\sqrt{2})\epsilon +  d_H\left(K^{(1)}_i, K^{(2)}_i\right).
\]
If $U := \sum^k_{i=1} V_i \in \mathfrak{A}$ then $U$ is a unitary operator such that 
\[
\left\|N_1 - U^*N_2U\right\| < (4 + 2\sqrt{2})\epsilon + \max_{i \in \{1,\ldots, n\}} d_H\left(K^{(1)}_i, K^{(2)}_i\right).
\]
Hence the result follows.
\end{proof}
\begin{cor}
\label{hausdorffdistanceforsimilarity}
Let $\mathfrak{A}$ be a unital, simple, purely infinite C$^*$-algebra and let $N_1, N_2 \in \mathfrak{A}$ be normal operators such that $\Gamma(N_1)$ and $\Gamma(N_2)$ are trivial. Suppose
\begin{enumerate}
	\item  $\sigma(N_2) \subseteq \sigma(N_1)$ and
	\item $N_1$ and $N_2$ have equivalent common spectral projections.
\end{enumerate}
Then
\[
dist(\mathcal{U}(N_1), \mathcal{U}(N_2)) =  d_H\left(\sigma(N_1), \sigma(N_2)\right).
\]
\end{cor}
\begin{proof}
One inequality follows from Proposition \ref{lowerbounddistantofunitaryorbits}.  The two conditions listed in this corollary imply the suppositions of Proposition \ref{unionwithinepsilondistanceresult} are satisfied for every choice of $\epsilon$.  Hence
\[
dist(\mathcal{U}(N_1), \mathcal{U}(N_2)) \leq (4 + 2\sqrt{2})\epsilon + d_H\left(\sigma(N_1), \sigma(N_2)\right).
\]
for all $\epsilon > 0$.
\end{proof}
We have made use of the equivalence of certain spectral projections in the creation of all of the above bounds.  To illustrate the necessity of these assumptions, we note the following example.
\begin{exam}
Let $P$ and $Q$ be non-trivial projections in $\mathcal{O}_3$ with $[P]_0 \neq [Q]_0$.  Then $\sigma(P) = \sigma(Q)$ yet $dist(\mathcal{U}(P), \mathcal{U}(Q)) \geq 1$ or else $P$ and $Q$ would be Murray-von Neumann equivalent (see {\cite[Proposition 2.2.4]{RLL}} and \cite[Proposition 2.2.7]{RLL}).
\end{exam}
In particular, we have the following quantitative version of the above example.
\begin{prop}
Let $\mathfrak{A}$ be a unital C$^*$-algebra, let $N_1, N_2 \in \mathfrak{A}$ be normal operators, and let $f : \mathbb{C} \to \mathbb{C}$ be a function that is analytic on an open neighbourhood $U$ of $\sigma(N_1) \cup \sigma(N_2)$ with $f(U) \subseteq \{0,1\}$.  Let $\gamma$ be a compact, rectifiable curve inside $U$ with $(\sigma(N_1) \cup \sigma(N_2)) \cap \gamma =\emptyset$, $Ind_\gamma(z) \in \{0,1\}$ for all $z \in \mathbb{C} \setminus \gamma$, $Ind_\gamma(z) = 1$ for all $z \in \sigma(N_1) \cup \sigma(N_2)$, and $\{z \in \mathbb{C} \, \mid\, Ind_\gamma(z) \neq 0\}\subseteq U$.  If $f(N_1)$ and $f(N_2)$ are not Murray-von Neumann equivalent then
\[
dist(\mathcal{U}(N_1), \mathcal{U}(N_2)) \geq \frac{2\pi}{l_0(\gamma) \sup_{z \in \gamma} \left\|(z I_\mathfrak{A} - N_1)^{-1}\right\|\left\|(z I_\mathfrak{A} - N_2)^{-1}\right\|}
\]
where $l_0(\gamma)$ is the length of $\gamma$ in the regions where $f(z) = 1$.
\end{prop}
\begin{proof}
By the proof of Lemma \ref{analytic}, we know that $\left\|f(N_1) - Uf(N_2)U^*\right\|$ is at most
\[
\frac{l_0(\gamma) \left\|N_1 - UN_2U^*\right\|}{2\pi} \sup_{z \in \gamma} \left\|(z I_\mathfrak{A} - N_1)^{-1}\right\|\left\|(z I_\mathfrak{A} - N_2)^{-1}\right\|
\]
for all unitaries $U$ in $\mathfrak{A}$.  Since $f(N_1)$ and $f(N_2)$ are not Murray-von Neumann equivalent, $f(N_1)$ and $Uf(N_2)U^*$ are not Murray-von Neumann equivalent so
\[
1 \leq \left\|f(N_1) - Uf(N_2)U^*\right\|
\]
by {\cite[Proposition 2.2.5]{RLL}} and {\cite[Proposition 2.2.7]{RLL}}.  Hence the result follows.
\end{proof}
To complete this section, we desire to examine the distance between unitary orbits of normal operators with non-trivial index function.  Unfortunately, as this problem is not complete even for the Calkin algebra and due to the technical restraints illustrated above, a complete description of the distance between unitary orbits will not be given.  In particular our goal is to generalize Corollary \ref{hausdorffdistanceforsimilarity} to a sufficient degree to be used in Section \ref{sec:CLOSESIMORBITS}.  We shall proceed with this goal by attempting to modify {\cite[Theorem 1.4]{Da1}} via an application of {\cite[Theorem 1.7]{Dad}}.
\par
As in the proof of {\cite[Theorem 1.4]{Da1}}, we will need a notion of direct sums inside unital, simple, purely infinite C$^*$-algebras.  This leads us to the following construction.
\begin{lem}
\label{UHFdirectsum}
Let $\mathfrak{A}$ be a unital, simple, purely infinite C$^*$-algebra, let $V \in \mathfrak{A}$ be a non-unitary isometry, and let $P := VV^*$.  Then there exists a unital embedding of the $2^\infty$-UHF C$^*$-algebra $\mathfrak{B} := \overline{\bigcup_{\ell\geq 1} \mathcal{M}_{2^\ell}(\mathbb{C})}$ into $(I_\mathfrak{A} - P)\mathfrak{A}(I_\mathfrak{A} - P)$ such that $[Q]_0 = 0$ in $\mathfrak{A}$ for every projection $Q \in \mathfrak{B}$.
\end{lem}
\begin{proof}
Let $P_0 := I_\mathfrak{A} - P$.  Since $\mathfrak{A}$ is a unital, simple, purely infinite C$^*$-algebra, there exists a projection $P_1 \in \mathfrak{A}$ such that $P_0$ and $P_1$ are Murray-von Neumann equivalent and $0 < P_1 < P_0$.  Let $P_2 := P_0 - P_1$ which is a non-trivial projection.  Note $[P_0]_0 = 0$ in $\mathfrak{A}$ by {\cite[Theorem 1.4]{Cu}}.  Hence
\[
[P_1]_0 = [P_0]_0 = 0 = [P_1 + P_2]_0 = [P_1]_0 + [P_2]_0 = [P_2]_0.
\]
Thus $P_1$ and $P_2$ are Murray-von Neumann equivalent in $\mathfrak{A}$ by {\cite[Theorem 1.4]{Cu}}.  Thus, since $P_1, P_2 \leq P_0$, $P_1$ and $P_2$ are Murray-von Neumann equivalent in $P_0\mathfrak{A}P_0$.
\par
For $q \in \{1,2\}$ let $V_q \in P_0\mathfrak{A}P_0$ be an isometry such that $V_qV_q^* = P_q$.  Then it is not difficult to see for each $\ell \in \mathbb{N}$ that
\[
\mathfrak{B}_\ell := *\mbox{-alg}\left(\{V_{i_1}V_{i_2}\cdots V_{i_\ell}V_{j_\ell}^* \cdots V_{j_2}^* V_{j_1}^* \, \mid \, i_1, i_2, \ldots, i_\ell, j_1, j_2\ldots, j_\ell \in \{1,2\}\}\right)
\]
is a C$^*$-subalgebra of $P_0\mathfrak{A}P_0$ containing $P_0$ that is isomorphic to $\mathcal{M}_{2^\ell}(\mathbb{C})$.  Moreover, it is clear that $\mathfrak{B}_\ell \subseteq \mathfrak{B}_{\ell+1}$ for all $\ell \in \mathbb{N}$ and
\[
\{V_{i_1}V_{i_2}\cdots V_{i_\ell}V_{j_\ell}^* \cdots V_{j_2}^* V_{j_1}^* \, \mid \, i_1, i_2, \ldots, i_\ell, j_1, j_2\ldots, j_\ell \in \{1,2\}\}
\]
are matrix units for $\mathfrak{B}_\ell$ in such a way that $\mathfrak{B} := \overline{\bigcup_{\ell\geq 1} \mathfrak{B}_\ell}$ is the $2^\infty$-UHF C$^*$-algebra.  Notice every rank one projection in $\mathfrak{B}_\ell$ is Murray-von Neumann equivalent in $\mathfrak{B}_\ell$ (and thus in $P_0\mathfrak{A}P_0$) to the rank one matrix unit $(V_1)^\ell(V_1^*)^\ell$ which is Murray-von Neumann equivalent in $\mathfrak{A}$ to $P_0$.  Therefore $[Q]_0 = [P_0]_0 = 0$ in $\mathfrak{A}$ for every rank one projection $Q \in \mathfrak{B}_\ell$.  Hence $[Q]_0 = 0$ in $\mathfrak{A}$ for every non-zero projection $Q \in \mathfrak{B}_\ell$.  However, if $Q \in \mathfrak{B}$ is a non-zero projection, it is easy to see that there exists an $\ell \in \mathbb{N}$ and a non-zero projection $Q_0 \in \mathfrak{B}_\ell$ such that $\left\|Q - Q_0\right\| < \frac{1}{2}$.  Hence $Q$ and $Q_0$ are Murray-von Neumann equivalent in $\mathfrak{A}$ by Lemma \ref{simorbitofprojections}.  Thus $[Q]_0 = [Q_0]_0 = 0$ as desired.
\end{proof}
We will need the following two results to generalize {\cite[Theorem 1.4]{Da1}} to our desired context.
\begin{lem}
\label{UHFanyspectrum}
Let $\mathfrak{B} := \overline{\bigcup_{\ell\geq 1} \mathcal{M}_{2^\ell}(\mathbb{C})}$ be the $2^\infty$-UHF C$^*$-algebra.  If $X \subseteq \mathbb{C}$ is compact, there exists a normal operator $N \in \mathfrak{B}$ such that $\sigma(N) = X$.
\end{lem}
\begin{proof}
For each $n \in \mathbb{N}$ there exists a $\frac{1}{2^n}$-net $X_n = \{x_{j,n}\}^{k_n}_{j=1} \subseteq X$ of $X$.  By choosing $\ell_1$ such that $k_1 \leq 2^{\ell_1}$, there exists a normal operator $N_1 \in\mathcal{M}_{2^{\ell_1}}(\mathbb{C})$ such that $\sigma(N_1) = X_1$.  Since $X_1$ is a 1-net for $X$, it is easy to see that by choosing $\ell_2 > \ell_1$ sufficiently large there exists a normal operator $N_2 \in \mathcal{M}_{2^{\ell_2}}(\mathbb{C})$ such that $\sigma(N_2) = X_1 \cup X_2$ and the distance between $N_2$ and the canonical embedding of $N_1$ into $\mathcal{M}_{2^{\ell_2}}(\mathbb{C})$ is at most 1 (that is, by embedding $N_1$ into a large $\mathcal{M}_{2^\ell}(\mathbb{C})$ we can ensure that each element of $\sigma(N_1)$ in the spectrum of the embedded matrix has large enough multiplicity in order to approximate the elements of $X_1 \cup X_2$ separately).  By repeating this procedure, we obtain a strictly increasing sequence $(\ell_n)_{n\geq1}$ of natural numbers and normal operators $N_n \in \mathcal{M}_{2^{\ell_n}}(\mathbb{C})$ such that $\sigma(N_n) = \bigcup^n_{j=1} X_j$ and $(N_n)_{n\geq 1}$ is a Cauchy sequence in $\mathfrak{B}$.
\par
Let $N := \lim_{n\to\infty} N_n$.  Since $X_j \subseteq \sigma(N_n)$ for all $n\geq j$, $X_j \subseteq \sigma(N)$.  Therefore, as $\bigcup_{j\geq 1} X_j$ is dense in $X$, $X \subseteq \sigma(N)$.  The inclusion $\sigma(N) \subseteq X$ follows trivially from Lemma \ref{densespectra2}.
\end{proof}
\begin{lem}[see {\cite[Lemma 1.2]{Cu}}]
\label{connectedcomponenttechnicallity}
Let $\mathfrak{A}$ be a unital, simple, purely infinite C$^*$-algebra, let $V \in \mathfrak{A}$ be an isometry, and let $U \in \mathfrak{A}$ be a unitary.  Then $[U]_1 = [VUV^* + (I_\mathfrak{A} - VV^*)]_1$.
\end{lem}
\begin{thm}
\label{nonzeroindexdistance}
Let $\mathfrak{A}$ be a unital, simple, purely infinite C$^*$-algebra and let $N,M \in \mathfrak{A}$ be normal operators such that
\begin{enumerate}
	\item $\sigma(M) \subseteq \sigma(N)$,
	\item $\sigma(M)$ intersects every connected component of $\sigma(N)$,
	\item $\Gamma(M)(\lambda) = \Gamma(N)(\lambda)$ for all $\lambda \notin \sigma(N)$, and
	\item $N$ and $M$ have equivalent common spectral projections.
\end{enumerate}
Then
\[
dist(N, M) = d_H(\sigma(N), \sigma(M)).
\]
\end{thm}
\begin{proof}
One inequality follows from Proposition \ref{lowerbounddistantofunitaryorbits}.  Since $\mathfrak{A}$ is a unital, simple, purely infinite C$^*$-algebra, there exists a non-unitary isometry $V \in \mathfrak{A}$.  Let $P := VV^*$, let $\mathfrak{C}:= (I_\mathfrak{A} - P)\mathfrak{A}(I_\mathfrak{A} - P)$, and let $\mathfrak{B}$ be the unital copy of the $2^\infty$-UHF C$^*$-algebra in $\mathfrak{C}$ given by Lemma \ref{UHFdirectsum}.  By Lemma \ref{UHFanyspectrum} there exists normal operators $N_0, M_0 \in \mathfrak{B}$ such that $\sigma(N_0) = \sigma(N)$ and $\sigma(M_0) = \sigma(M)$.  
\par
Let $N' := VMV^* + N_0$ and let $M' := VMV^* + M_0$ which are clearly normal operators as $V$ is an isometry.  We will demonstrate that $N' \in \overline{\mathcal{U}(N)}$ and $M' \in \overline{\mathcal{U}(M)}$ by appealing to {\cite[Theorem 1.7]{Dad}}.  Notice that $\sigma(N') = \sigma(M) \cup \sigma(N_0) = \sigma(N)$ as $V$ is an isometry.  Furthermore if $f : \mathbb{C} \to \mathbb{C}$ is a function that is analytic on an open neighbourhood $U$ of $\sigma(N)$ with $f(U) \subseteq \{0,1\}$ then
\[
f(N') = f(VMV^*) + f(N_0) = Vf(M)V^* + f(N_0).
\]
If $f(M) = 0$ then $f(N) = 0$ as $f(M)$ and $f(N)$ are Murray-von Neumann equivalent.  This implies $f$ is zero on $\sigma(N)$ and thus $f(N') = f(N_0) = 0 = f(N)$.  If $f(M) \neq 0$ then $f(N') \neq 0$ and
\[
[f(N')]_0 = [Vf(M)V^*]_0 + [f(N_0)]_0 = [f(M)]_0 = [f(N)]_0
\]
as $f(N_0) \in \mathfrak{B}$ and as every projection in $\mathfrak{B}$ is trivial in the $K_0$-group of $\mathfrak{A}$ by Lemma \ref{UHFdirectsum}.  In any case $f(N')$ and $f(N)$ are Murray-von Neumann equivalent.  Furthermore, since $\mathfrak{B}^{-1}_0 = \mathfrak{B}^{-1}$ as $\mathfrak{B}$ is a UHF C$^*$-algebra, we notice for any $\lambda \notin \sigma(N)$ that $\lambda I_\mathfrak{A} - N'$ is in the same component of $\mathfrak{A}^{-1}$ as
\[
V(\lambda I_\mathfrak{A} - M)V^* + (\lambda I_\mathfrak{A} - P)
\]
which is in the same connected component of $\mathfrak{A}^{-1}$ as $\lambda I_\mathfrak{A} - M$ by Lemma \ref{connectedcomponenttechnicallity}.  Therefore, since $\Gamma(M)(\lambda) = \Gamma(N)(\lambda)$ for all $\lambda \notin \sigma(N)$ by assumption, we obtain that $\Gamma(N') = \Gamma(N)$.  Therefore $N$ and $N'$ are approximately unitarily equivalent in $\mathfrak{A}$ by {\cite[Theorem 1.7]{Dad}}.  Similarly $M$ and $M'$ are approximately unitarily equivalent in $\mathfrak{A}$ by {\cite[Theorem 1.7]{Dad}}.
\par
Hence it is easy to see for any unitary $U \in \mathfrak{C}$ that
\[
dist(\mathcal{U}(N), \mathcal{U}(M)) \leq \left\|(P + U)N'(P+U)^* - M'\right\| = \left\|UN_0U^* - M_0\right\|.
\]
However, since $\mathfrak{C}$ is a unital, simple, purely infinite C$^*$-algebra and $N_0, M_0 \in \mathfrak{C}$ are in the unital inclusion of the UHF C$^*$-algebra $\mathfrak{B}$ in $\mathfrak{C}$, it is easy to see that $\Gamma(N_0)$ and $\Gamma(M_0)$ are trivial (when viewed as elements of $\mathfrak{C}$).  Moreover, since $\sigma(M)$ intersects every connected component of $\sigma(N)$, since $\sigma(M_0) = \sigma(M) \subseteq \sigma(N) = \sigma(N_0)$, and since any two non-zero projections in $\mathfrak{B} \subseteq \mathfrak{C}$ are Murray-von Neumann equivalent, the hypotheses of Corollary \ref{hausdorffdistanceforsimilarity} are satisfied for $N_0$ and $M_0$ in $\mathfrak{C}$.  Hence for any $\epsilon > 0$ there exists a unitary $U \in \mathfrak{C}$ such that 
\[
\left\|UN_0U^* - M_0\right\| \leq \epsilon + d_H(\sigma(N_0), \sigma(M_0)) = \epsilon + d_H(\sigma(N), \sigma(M)).
\]
Hence 
\[
dist(\mathcal{U}(N), \mathcal{U}(M)) \leq d_H(\sigma(N), \sigma(M))
\]
as desired.
\end{proof}
To complete this section we note that the proof of Theorem \ref{nonzeroindexdistance} can be adapted to obtain additional results provided there is a method for matching spectral projections.  In particular {\cite[Theorem 1.4]{Da1}} clearly generalizes to the following results.
\begin{prop}
\label{upgradedav}
Let $\mathfrak{A}$ be a unital, simple, purely infinite C$^*$-algebra with trivial $K_0$-group.  If $N_1, N_2 \in \mathfrak{A}$ are normal operators then 
\[
dist(\mathcal{U}(N_1), \mathcal{U}(N_2)) \leq 2 \rho(N_1, N_2)
\]
where $\rho(N_1, N_2)$ is as defined in Definition \ref{rho}.
\end{prop}
\begin{proof}
Since $\mathfrak{A}$ is a unital, simple, purely infinite C$^*$-algebra, there exists a non-unitary isometry $V \in \mathfrak{A}$.  Let $P := VV^*$, let $\mathfrak{C}:= (I_\mathfrak{A} - P)\mathfrak{A}(I_\mathfrak{A} - P)$, and let $\mathfrak{B}$ be the unital copy of the $2^\infty$-UHF C$^*$-algebra in $\mathfrak{C}$ given by Lemma \ref{UHFdirectsum}.  
\par
Let $X$ be the compact set that is the union of $\sigma(N_1)$, $\sigma(N_2)$, and 
\[
\{\lambda \in \mathbb{C} \, \mid \, \lambda \notin \sigma(N_1) \cup \sigma(N_2), \Gamma(N_1)(\lambda) \neq \Gamma(N_2)(\lambda)\}.
\]
By Lemma \ref{UHFanyspectrum} there exists a normal operator $N' \in \mathfrak{B}$ such that $\sigma(N') = X$.  Therefore, if
\[
M := VN_1V^* + N'
\]
then $M$ is a normal operator in $\mathfrak{A}$ such that $\sigma(M) = X$ and $\Gamma(M)(\lambda) = \Gamma(N_1)(\lambda) = \Gamma(N_2)(\lambda)$ for all $\lambda \notin X$.  Therefore it suffices to show for any $q \in \{1,2\}$ that 
\[
dist(\mathcal{U}(N_q), \mathcal{U}(M)) \leq \rho(N_1, N_2).
\]
\par
By the definition of $\rho$ we see that
\[
\rho(N_q, M) = d_H(\sigma(N_q), \sigma(M)) \leq \rho(N_1, N_2).
\]
Furthermore, by applying Lemma \ref{UHFanyspectrum}, there exists normal operators $N_0, M_0 \in \mathfrak{B}$ such that $\sigma(N_0) = \sigma(N_q)$ and $\sigma(M_0) = \sigma(M)$.  As in the proof of Theorem \ref{nonzeroindexdistance}, we see that $VN_qV^* + N_0 \in \overline{\mathcal{U}(N_q)}$ and $VN_qV^* + M_0 \in \overline{\mathcal{U}(M)}$.  Hence it is easy to see that for any unitary $U \in \mathfrak{C}$ that
\[
\begin{array}{rcl}
dist(\mathcal{U}(N_q), \mathcal{U}(M)) &\leq&  \left\|(P + U)(VN_qV^* + N_0)(P+U)^* - (VN_qV^* + M_0)\right\| \\
 &=&  \left\|UN_0U^* - M_0\right\|.
 \end{array} 
\]
Thus, as in the proof of Theorem \ref{nonzeroindexdistance}, for any $\epsilon > 0$ there exists a $U \in \mathfrak{C}$ such that 
\[
\left\|UN_0U^* - M_0\right\| \leq \epsilon + d_H(\sigma(N_1), \sigma(M)) \leq \epsilon + \rho(N_1, N_2).
\]
Hence the result follows.
\end{proof}
\begin{prop}
Let $\mathfrak{A}$ be a unital, simple, purely infinite C$^*$-algebra.  If $N_1, N_2 \in \mathfrak{A}$ are normal operators with common spectral projections then 
\[
dist(\mathcal{U}(N_1), \mathcal{U}(N_2)) \leq 2 \rho(N_1, N_2).
\]
\end{prop}
\begin{proof}
The proof of this result follows the proof of Proposition \ref{upgradedav} where we note $N_1$ and $N_2$ having common spectral projections implies that $N_1$ and $M$ have common spectral projections and $N_2$ and $M$ have common spectral projections.  This facilitates the proof that $VN_qV^* + N_0 \in \overline{\mathcal{U}(N_q)}$ and $VN_qV^* + M_0 \in \overline{\mathcal{U}(M)}$ and thus the rest of the proof follows.
\end{proof}

\section{Closed Similarity Orbits of Normal Operators}
\label{sec:CLOSESIMORBITS}

A complete classification of when a normal operator is in the closed similarity orbit of another normal operator in the Calkin algebra was obtained in {\cite[Theorem 2]{AHV}}.  We state this result to facilitate the comparison with the results of this section.
\begin{thm}[{\cite[Theorem 2]{AHV}}, see {\cite[Theorem 9.3]{AFHV} for a proof}]
\label{calkinsim}
Let $N$ and $M$ be normal operators in the Calkin algebra.  Then $N \in \overline{\mathcal{S}(M)}$ if and only if
\begin{enumerate}
	\item $\sigma_e(M) \subseteq \sigma_e(N)$,
	\item each component of $\sigma_e(N)$ intersects $\sigma_e(M)$,
	\item the Fredholm index of $\lambda I - M$ and $\lambda I - N$ agree for all $\lambda \notin \sigma_e(N)$, and
	\item if $\lambda \in \sigma_e(N)$ is not isolated in $\sigma_e(N)$, the component of $\lambda$ in $\sigma_e(N)$ contains some non-isolated point of $\sigma_e(M)$.	
\end{enumerate}
\end{thm}
As the Calkin algebra is a unital, simple, purely infinite C$^*$-algebra, in this section we endeavour to use the results of Section \ref{sec:DISTANCEUNITARYORBITS} and results from \cite{Sk} to generalize the above theorem.  In addition, we will obtain a generalization of the above theorem to type III factors with separable predual.  The two main results of this section are similar in proof but pose slight technical differences and thus are listed separately.
\begin{thm}
\label{main2}
Let $\mathfrak{A}$ be a unital, simple, purely infinite C$^*$-algebra and let $N,M \in \mathfrak{A}$ be normal operators.  Then $N \in \overline{\mathcal{S}(M)}$ if and only if
\begin{enumerate}
	\item $\sigma(M) \subseteq \sigma(N)$,
	\item each component of $\sigma(N)$ intersects $\sigma(M)$,
	\item $\Gamma(N)(\lambda) = \Gamma(M)(\lambda)$ for all $\lambda \notin \sigma(N)$,
	\item if $\lambda \in \sigma(N)$ is not isolated in $\sigma(N)$, the component of $\lambda$ in $\sigma(N)$ contains some non-isolated point of $\sigma(M)$, and
	\item $N$ and $M$ have equivalent common spectral projections.
\end{enumerate}
\end{thm}
\begin{thm}
\label{main}
Let $\mathfrak{A}$ be a unital C$^*$-algebra with the following properties;
\begin{enumerate}
	\item $\mathfrak{A}$ has property weak (FN),
	\item every non-zero projection in $\mathfrak{A}$ is properly infinite, and
	\item any two non-zero projections in $\mathfrak{A}$ are Murray-von Neumann equivalent.
\end{enumerate}
(For example, $\mathcal{O}_2$ and every type III factor with separable predual.)
\par
Let $N,M \in \mathfrak{A}$ be normal operators such that $\lambda I_\mathfrak{A} - M \in \mathfrak{A}^{-1}_0$ for all $\lambda \notin \sigma(M)$.  Then $N \in \overline{\mathcal{S}(M)}$ if and only if
\begin{enumerate}
	\item $\sigma(M) \subseteq \sigma(N)$,
	\item each component of $\sigma(N)$ intersects $\sigma(M)$,
	\item $\lambda I_\mathfrak{A} - N \in \mathfrak{A}^{-1}_0$ for all $\lambda \notin \sigma(N)$, and
	\item if $\lambda \in \sigma(N)$ is not isolated in $\sigma(N)$, the component of $\lambda$ in $\sigma(N)$ contains some non-isolated point of $\sigma(M)$.
\end{enumerate}
\end{thm}
Note if $N \in \overline{\mathcal{S}(M)}$ then the first two conditions must hold by discussions from the beginning of Section \ref{sec:CLOSEDUNITARYORBITS} and the third condition follows from Lemma \ref{ccoits}.  The fifth condition of Theorem \ref{main2} is necessary by Lemma \ref{analytic} and Lemma \ref{simorbitofprojections}.
\par
To see that the fourth conclusion is necessary, let $K_\lambda$ be the connected component of $\sigma(N)$ containing $\lambda$.  We note that if $K_\lambda$ is not isolated in $\sigma(N)$ (that is, every open neighbourhood of $K_\lambda$ intersects a different connected component of $\sigma(N)$) then the first two conditions imply that $\sigma(M) \cap K_\lambda$ contains a cluster point of $\sigma(M)$.  Otherwise if $K_\lambda$ is isolated in $\sigma(N)$, the characteristic function $\chi_{K_\lambda}$ of $K_\lambda$ can be extended to an analytic function on a neighbourhood of $\sigma(N)$.  Thus Lemma \ref{analytic} implies $\chi_{K_\lambda}(N) \in \overline{\mathcal{S}(\chi_{K_\lambda}(M))}$.  If $\sigma(M) \cap K_\lambda$ does not contain a cluster point of $\sigma(M)$ then  $\chi_{K_\lambda}(M)$ must have finite spectrum.  Hence there exists a non-zero polynomial $p$ such that $p(\chi_{K_\lambda}(M)) = 0$.  Clearly this implies $p(T) = 0$ for all $T \in \overline{\mathcal{S}(\chi_{K_\lambda}(M))}$ so $p(\chi_{K_\lambda}(N)) = 0$.  Since $K_\lambda$ is a connected, compact subset of $\sigma(N)$ that is not a singleton, this is impossible.  Hence the fourth condition is necessary.  An alternative proof of the necessity of the fourth condition may be obtained by considering the separable C$^*$-algebra generated by $N$, $M$, and a countable number of invertible elements, by taking an infinite direct sum of a faithful representation of this C$^*$-algebra on a separable Hilbert space, and by appealing to property (e) of {\cite[Theorem 1]{BH}}.
\par
By applying Theorem \ref{main2} in conjunction with {\cite[Theorem 1.7]{Dad}}, the following result is easily obtained.
\begin{cor}
\label{jointsimimpliesaue}
Let $\mathfrak{A}$ be a unital, simple, purely infinite C$^*$-algebra and let $N_1, N_2 \in \mathfrak{A}$ be normal operators.  If $N_1 \in \overline{\mathcal{S}(N_2)}$ and $N_2 \in \overline{\mathcal{S}(N_1)}$ then $N_1 \sim_{au} N_2$.
\end{cor}
To begin the proofs of Theorem \ref{main2} and Theorem \ref{main}, we note the following trivial result about similarity of operators in C$^*$-algebras.
\begin{lem}
\label{2by2sim}
Let $\mathfrak{A}$ be a unital C$^*$-algebra, let $P \in \mathfrak{A}$ be a non-trivial projection, let $Z \in (I_\mathfrak{A} - P)\mathfrak{A}(I_\mathfrak{A} - P)$, and let $X \in \mathfrak{A}$ be such that $PX(I_\mathfrak{A} - P) = X$.  If $\lambda \notin \sigma_{(I_\mathfrak{A} - P)\mathfrak{A}(I_\mathfrak{A} - P)}(Z)$ then
\[
\lambda P + X + Z \sim \lambda P + Z.
\]
\end{lem}
\begin{proof}
Note that if $Y := X(\lambda (I_\mathfrak{A} - P) - Z)^{-1}$ then
\[
T := I_\mathfrak{A} + Y
\]
is invertible with
\[
T^{-1} = I_\mathfrak{A} - Y.
\]
A trivial computation shows
\[
T (\lambda P + X + Z) T^{-1} = \lambda P + Z.
\]
\end{proof}
\begin{cor}
\label{nbynsim}
Let $\mathfrak{A}$ be a unital C$^*$-algebra, let $n \in \mathbb{N}$, let $\lambda_1, \ldots, \lambda_n$ be distinct complex scalars, let $\{P_j\}^n_{j=1} \subseteq \mathfrak{A}$ be a set of non-trivial orthogonal projections with $\sum^n_{j=1} P_j = I_\mathfrak{A}$, and let $\{A_{i,j}\}^n_{i,j=1} \subseteq \mathfrak{A}$ be such that $A_{i,j} = 0$ if $i \geq j$ and $P_iA_{i,j}P_j = A_{i,j}$ for all $i < j$.  Then
\[
\sum^n_{j=1} \lambda_j P_j + \sum^n_{i,j=1} A_{i,j} \sim \sum^n_{j=1} \lambda_j P_j.
\]
\end{cor}
\begin{proof}
By applying Lemma \ref{2by2sim} with $P := P_1$, $Z := \sum^n_{j=1} \lambda_j P_j + \sum^n_{i,j=2} A_{i,j}$ (it is elementary to show that $\sigma_{(I_\mathfrak{A} - P)\mathfrak{A} (I_\mathfrak{A} - P)}(Z) = \{\lambda_2,\ldots, \lambda_{n}\}$ so $\lambda_1 \notin \sigma(Z)$ by assumption), and $X:= \sum^n_{j=1} A_{1,j}$, we obtain that
\[
\sum^n_{j=1} \lambda_j P_j + \sum^n_{i,j=1} A_{i,j} \sim \sum^n_{j=1} \lambda_j P_j + \sum^n_{i,j=2} A_{i,j}.
\]
The result then proceeds by induction by considering the unital C$^*$-algebra $(I_\mathfrak{A} - P_1)\mathfrak{A} (I_\mathfrak{A} - P_1)$.
\end{proof}
To begin the proof of Theorem \ref{main2}, we first show that a `direct sum' of a normal operator and a nilpotent operator is in the similarity orbit of the normal operator.  The idea of this result is based on {\cite[Lemma 5.3]{He}}.
\begin{lem}
\label{directsumnil2}
Let $\mathfrak{A}$ be a unital, simple, purely infinite C$^*$-algebra, let $M \in \mathfrak{A}$, let $V \in \mathfrak{A}$ be a non-unitary isometry, let $P := VV^*$, and let $\mathfrak{B}:= \overline{\bigcup_{\ell\geq1} \mathcal{M}_{2^\ell}(\mathbb{C})}$ be the unital copy of the $2^\infty$-UHF C$^*$-algebra in $\mathfrak{C}$ given by Lemma \ref{UHFdirectsum}.  Suppose $\mu$ is a cluster point of $\sigma(M)$ and $Q \in \mathcal{M}_{2^\ell}(\mathbb{C}) \subseteq \mathfrak{B}$ is a nilpotent matrix for some $\ell \in \mathbb{N}$.  Then $VMV^* + \mu (I_\mathfrak{A} - P) + Q \in \overline{\mathcal{S}(M)}$.
\end{lem}
\begin{proof}
Since $Q \in \mathcal{M}_{2^\ell}(\mathbb{C}) \subseteq \mathfrak{B}$ is a nilpotent matrix, $Q$ is unitarily equivalent to a strictly upper triangular matrix.  Thus we can assume $Q$ is strictly upper triangular.  By our assumptions on $\mu$ there exists a sequence $(\mu_j)_{j\geq1}$ of distinct scalars contained in $\sigma(M)$ that converges to $\mu$.  For each $q \in \mathbb{N}$ let 
\[
T_q:= diag(\mu_q, \mu_{q+1}, \ldots \mu_{q+2^\ell-1}) \in \mathcal{M}_{2^\ell}(\mathbb{C}) \subseteq \mathfrak{B}
\]
be the diagonal matrix with $\mu_q$, $\ldots$, $\mu_{q+2^\ell-1}$ along the diagonal.  
\par
Let $M_q := VMV^* + T_q \in \mathfrak{A}$.  As in the proof of Theorem \ref{nonzeroindexdistance}, it is easy to see by {\cite[Theorem 1.7]{Dad}} that $M_q$ is approximately unitarily equivalent to $M$ for each $q \in \mathbb{N}$.  Hence
\[
M \sim_{au}  M_q \sim M \oplus \left( \bigoplus^{n}_{k=1} T_q + Q\right)
\]
by Lemma \ref{nbynsim}.  Since $\lim_{q\to \infty} T_q + Q = \mu (I_\mathfrak{A} - P) + Q$, the result follows.
\end{proof}
Subsequently we have our next stepping-stone which based on {\cite[Corollary 5.5]{He}}.
\begin{lem}
\label{fillinspec2}
Let $\mathfrak{A}$ be a unital, simple, purely infinite C$^*$-algebra.  Let $N,M \in \mathfrak{A}$ be normal operators and write $\sigma(N) = K_1 \cup K_2$ where $K_1$ and $K_2$ are disjoint compact sets with $K_1$ connected.  Suppose 
\begin{enumerate}
	\item $\sigma(M) = K'_1 \cup K_2$ where $K'_1 \subseteq K_1$, 
	\item $\Gamma(N)(\lambda) = \Gamma(M)(\lambda)$ for all $\lambda \notin \sigma(N)$, and
	\item $N$ and $M$ have equivalent common spectral projections.
\end{enumerate}
If $K'_1$ contains a cluster point of $\sigma(M)$ then $N \in \overline{\mathcal{S}(M)}$.
\end{lem}
\begin{proof}
If $K_1$ is a singleton,  $K'_1 = K_1$ as $K'_1$ is non-empty.  Thus $\sigma(M) = \sigma(N)$ so Theorem \ref{ausiuspi} implies $N$ and $M$ are approximately unitarily equivalent.
\par
Otherwise $K'_1$ is not a singleton.  Fix a non-unitary isometry $V \in \mathfrak{A}$ and $\epsilon > 0$.  Let $P := VV^*$ and let $\mathfrak{B}:= \overline{\bigcup_{\ell\geq1} \mathcal{M}_{2^\ell}(\mathbb{C})}$ be the unital copy of the $2^\infty$-UHF C$^*$-algebra in $(I_\mathfrak{A} - P)\mathfrak{A}(I_\mathfrak{A} - P)$ given by Lemma \ref{UHFdirectsum}.  By {\cite[Theorem 6.10]{Sk}} there exists a normal operator $T \in \mathfrak{B}$ with 
\[
\sigma(T) = \{z \in \mathbb{C} \, \mid \, |z| \leq \epsilon\},
\]
such that $T$ is a norm limit of nilpotent matrices from $\bigcup_{\ell\geq 1} \mathcal{M}_{2^\ell}(\mathbb{C}) \subseteq \mathfrak{B} \subseteq \mathfrak{A}$.  Let $\mu \in K'_1$ be any cluster point of $\sigma(M)$.  Lemma \ref{directsumnil2} implies that 
\[
VMV^* + \mu (I_\mathfrak{A} - P) + Q \in \overline{\mathcal{S}(M)}
\]
for every nilpotent matrix $Q \in \bigcup_{\ell\geq 1} \mathcal{M}_{2^\ell}(\mathbb{C}) \subseteq \mathfrak{B}$.   Since $T$ is a norm limit of nilpotent matrices from $\bigcup_{\ell\geq 1} \mathcal{M}_{2^\ell}(\mathbb{C})$, we obtain that 
\[
VMV^* + \mu (I_\mathfrak{A} - P) + T \in \overline{\mathcal{S}(M)}.
\]
\par
Let $M_1 := VMV^* + \mu (I_\mathfrak{A} - P) + T$.  As in the proof of Theorem \ref{nonzeroindexdistance}, it is easy to see that $M_1$ is a normal operator such that $\Gamma(M_1)(\lambda) = \Gamma(M)(\lambda) = \Gamma(N)(\lambda)$ for all $\lambda \notin \sigma(M_1) \cup \sigma(N)$ and $M_1$ and $N$ have equivalent common spectral projections.  
\par
Since $K_1$ is connected and $\sigma(M_1)$ contains an open neighbourhood around $\mu \in K_1$, we can repeat the above argument a finite number of times to obtain a normal operator $M_0 \in \overline{\mathcal{S}(M)}$ such that $\sigma(M_0) = K''_1 \cup K_2$ where $K''_1$ is connected, $K_1 \subseteq K''_1$,
\[
K''_1 \subseteq \{z \in \mathbb{C} \, \mid \, dist(z, K_1)\leq \epsilon\},
\]
$\Gamma(M_0)(\lambda) = \Gamma(N)(\lambda)$ for all $\lambda \notin \sigma(M_1) \cup \sigma(N)$, and $M_0$ and $N$ have equivalent common spectral projections.  Therefore Theorem \ref{nonzeroindexdistance} implies $dist(\mathcal{U}(N), \mathcal{U}(M_0)) = d_H(\sigma(N), \sigma(M_0)) \leq \epsilon$ so $dist(N, \mathcal{S}(M)) \leq \epsilon$.  Thus, as $\epsilon > 0$ was arbitrary, the result follows.
\end{proof}
We can now complete the proof of Theorem \ref{main2} using the above result.
\begin{proof}[Proof of Theorem \ref{main2}]
Let $N$ and $M$ satisfy the five conditions of Theorem \ref{main2}.  By applying Lemma \ref{fillinspec2} recursively a finite number of times, we can find a normal operator $M'$ such that $M' \in \overline{\mathcal{S}(M)}$, $\sigma(M')$ is $\sigma(M)$ unioned with a finite number of connected components of $\sigma(N)$, and $N$ and $M'$ satisfy the five conditions of Theorem \ref{main2} 
\par
Fix $\epsilon > 0$.  Since $\sigma(N)$ is compact, $\sigma(N)$ has a finite $\epsilon$-net.  Thus the normal operator $M'$ in the above paragraph can be selected with the additional requirement that $dist(\lambda, \sigma(M')) \leq 2\epsilon$ for all $\lambda \in \sigma(N)$.  By Theorem \ref{nonzeroindexdistance} $dist(N, \mathcal{U}(M')) \leq 2\epsilon$ so $dist(\mathcal{U}(N), \mathcal{S}(M)) \leq 2\epsilon$ as desired.
\end{proof}
Note that by using Corollary \ref{aueintrivialk1} instead of {\cite[Theorem 1.7]{Dad}} and Corollary \ref{hausdorffdistanceforsimilarity} instead of Theorem \ref{nonzeroindexdistance}, a proof of Theorem \ref{main2} that is independent of {\cite[Theorem 1.7]{Dad}} may be obtained for any unital, simple, purely infinite C$^*$-algebra with trivial $K_1$-group.  Similarly, using {\cite[Theorem 11.1]{BDF}} and {\cite[Theorem 1.4]{Da1}}, the proof of Theorem \ref{main2} is greatly simplified for the Calkin algebra and provides an alternate proof of Theorem \ref{calkinsim}.
\par
With the proof of Theorem \ref{main2} complete, we endeavour to prove Theorem \ref{main}.  As the proof of Theorem \ref{main2} relies on an embedding of the scalar matrices inside the C$^*$-algebra under consideration, we make the following definition.
\begin{defn}
Let $\mathfrak{A}$ be a unital C$^*$-algebra.  An operator $A \in \mathfrak{A}$ is said to be a scalar matrix in $\mathfrak{A}$ if there exists a finite dimensional C$^*$-algebra $\mathfrak{B}$ and a unital, injective $^*$-homomorphism $\pi : \mathfrak{B} \to \mathfrak{A}$ such that $A \in \pi(\mathfrak{B})$.
\end{defn}
The point of considering scalar matrices in the context of Theorem \ref{main} is the following.
\begin{prop}
\label{scalardisk2}
Let $\mathfrak{A}$ be a unital C$^*$-algebra with the three properties listed in Theorem \ref{main}.  If $N \in \mathfrak{A}$ is a normal operator with the closed unit disk as spectrum then $N$ is a norm limit of nilpotent scalar matrices from $\mathfrak{A}$.
\end{prop}
\begin{proof}
It is easy to see the second and third assumptions in Theorem \ref{main} imply that the $2^\infty$-UHF C$^*$-algebra has a unital, faithful embedding into $\mathfrak{A}$. Therefore, by {\cite[Theorem 6.10]{Sk}}, $\mathfrak{A}$ has a normal operator $N_0$ with the closed unit disk as spectrum that is a norm limit of nilpotent scalar matrices from $\mathfrak{A}$.  Since every two normal operators with spectrum equal to the closed unit disk are approximately unitarily equivalent by Corollary \ref{aueios} the result follows.
\end{proof}
Using the ideas contained in the proof of Lemma \ref{directsumnil2}, it is possible to prove the following.
\begin{lem}
\label{directsumnil}
Let $\mathfrak{A}$ be a unital C$^*$-algebra with the following conditions;
\begin{enumerate}
	\item there exits a unital, injective $^*$-homomorphism $\pi : \mathfrak{A} \oplus \mathfrak{A} \to \mathfrak{A}$, and
	\item if $N_1, N_2 \in \mathfrak{A}$ are normal operators with $\lambda I_\mathfrak{A} - N_q \in \mathfrak{A}^{-1}_0$ for all $\lambda \notin \sigma(N_q)$ and $q \in \{1,2\}$, $N_1 \sim_{au} N_2$ if and only if $\sigma(N_1) = \sigma(N_2)$.
	\end{enumerate}
Let $M \in \mathfrak{A}$ be a normal operator with $\lambda I_\mathfrak{A} - M \in \mathfrak{A}^{-1}_0$ for all $\lambda \notin \sigma(M)$, let $\mu \in \sigma(M)$ be a cluster point of $\sigma(M)$, and let $Q \in \mathfrak{A}$ be a nilpotent scalar matrix.  Then $\pi(M \oplus (\mu I + Q)) \in \overline{\mathcal{S}(M)}$.
\end{lem}
By using similar ideas to the proof of Theorem \ref{main2} and by using the following lemma, the proof of Theorem \ref{main} is also complete.
\begin{lem}
\label{fillinspec}
Let $\mathfrak{A}$ be a unital C$^*$-algebra with the three properties listed in Theorem \ref{main}.  Let $N,M \in \mathfrak{A}$ be normal operators with $\lambda I_\mathfrak{A} - N \in \mathfrak{A}^{-1}_0$ for all $\lambda \notin \sigma(N)$ and $\lambda I_\mathfrak{A} - M \in \mathfrak{A}^{-1}_0$ for all $\lambda \notin \sigma(M)$.  Let $\{K_\lambda\}_{\Lambda}$ be the connected components of $\sigma(N)$.  Suppose 
\[
\sigma(M) = \left(\bigcup_{\lambda \in \Lambda\setminus \{\lambda_0\}} K_\lambda\right) \cup K_0
\]
where $K_0 \subseteq K_{\lambda_0}$.  If $K_0$ contains a cluster point of $\sigma(M)$ then $N \in \overline{\mathcal{S}(M)}$.
\end{lem}
\begin{proof}
The proof of this lemma follows the proof of Lemma \ref{fillinspec2} by using direct sums instead of non-unitary isometries and an application of Proposition \ref{hausdorffdist} provided that Lemma \ref{directsumnil} applies.  Note that the second and third assumptions of Theorem \ref{main} imply that the first assumption of Lemma \ref{directsumnil} holds and Corollary \ref{aueios} implies that the second assumption of Lemma \ref{directsumnil} holds.
\end{proof}
With the proofs of Theorem \ref{main2} and Theorem \ref{main} complete, we will use said theorems to classify when a normal operator is a limit of nilpotents in these C$^*$-algebras.  Thus Corollary \ref{nilpotentinuspi} provides another proof (although a more complicated proof) to {\cite[Theorem 2.8]{Sk}}.  Moreover Corollary \ref{FNnil} has slightly weaker conditions to any result given in \cite{Sk} (that is, there should exists C$^*$-algebras satisfying the assumptions of the following theorem that are not studied in \cite{Sk} although the author is not aware of them).  However, we note the proof of {\cite[Theorem 2.8]{Sk}} can be adapted to this setting.  These proofs are based on the proof of {\cite[Proposition 5.6]{He}}.
\begin{cor}
\label{nilpotentinuspi}
Let $\mathfrak{A}$ be a unital, simple, purely infinite C$^*$-algebra.  A normal operator $N \in \mathfrak{A}$ is a norm limits of nilpotent operators from $\mathfrak{A}$ if and only if $0 \in \sigma(N)$, $\sigma(N)$ is connected, and $\Gamma(N)$ is trivial.
\end{cor}
\begin{proof}
The requirements that $\sigma(N)$ is connected and contains zero follows by {\cite[Lemma 1.3]{Sk}}.  The condition that $\Gamma(N)$ is trivial follows from {\cite[Lemma 2.7]{Sk}}.
\par
Suppose $N \in \mathfrak{A}$ is a normal operator such that $0 \in \sigma(N)$, $\sigma(N)$ is connected, and $\Gamma(N)$ is trivial.  Let $\epsilon > 0$ and fix a non-unitary isometry $V \in \mathfrak{A}$.  Let $P := VV^*$ and let $\mathfrak{B}:= \overline{\bigcup_{\ell\geq1} \mathcal{M}_{2^\ell}(\mathbb{C})}$ be the unital copy of the $2^\infty$-UHF C$^*$-algebra in $(I_\mathfrak{A} - P)\mathfrak{A}(I_\mathfrak{A} - P)$ given by Lemma \ref{UHFdirectsum}.  By {\cite[Theorem 6.10]{Sk}} there exists a normal operator $T \in \mathfrak{B}$ with 
\[
\sigma(T) = \{z \in \mathbb{C} \, \mid \, |z| \leq \epsilon\}
\]
such that $T$ is a norm limit of nilpotent matrices from $\bigcup_{\ell\geq 1} \mathcal{M}_{2^\ell}(\mathbb{C}) \subseteq \mathfrak{B} \subseteq \mathfrak{A}$.
\par
Let $M := VNV^* + T \in \mathfrak{A}$.  Clearly $M$ is a normal operator such that $\sigma(M) = \sigma(N) \cup \sigma(T)$, $M$ and $N$ have equivalent common spectral projections, and $\Gamma(M)$ is trivial as in the proof of Theorem \ref{nonzeroindexdistance}.  Therefore Corollary \ref{hausdorffdistanceforsimilarity} implies that
\[
dist(\mathcal{U}(N), \mathcal{U}(M)) \leq \epsilon.
\]
However, we note that $\Gamma(T)$ is trivial when we view $T$ as a normal element in $\mathfrak{A}$.  Moreover, as $\sigma(N)$ is connected and contains zero, $\sigma(M)$ is connected and contains $\sigma(T)$.  Thus Theorem \ref{main2} (where conditions (4) and (5) are easily satisfied) implies that $M \in \overline{\mathcal{S}(T)}$ so
\[
dist(N, \mathcal{S}(T)) \leq \epsilon.
\]
However, as $T$ is a norm limits of nilpotent operators from $\mathfrak{B} \subseteq \mathfrak{A}$, the above inequality implies $N$ is within $2\epsilon$ of a nilpotent operator from $\mathfrak{A}$.  Thus the proof is complete.
\end{proof}
\begin{cor}
\label{FNnil}
Let $\mathfrak{A}$ be a unital, separable C$^*$-algebra with the three properties listed in Theorem \ref{main}.  A normal operator $N \in \mathfrak{A}$ is a norm limits of nilpotent operators from $\mathfrak{A}$ if and only if $0 \in \sigma(N)$, $\sigma(N)$ is connected, and $\lambda I_\mathfrak{A} - N \in \mathfrak{A}^{-1}_0$ for all $\lambda \notin \sigma(N)$.
\end{cor}
\begin{proof}
The proof of this result follows the proof of Corollary \ref{nilpotentinuspi} by using direct sums instead of non-unitary isometries (as in Lemma \ref{directsumnil}), Proposition \ref{hausdorffdist} instead of Corollary \ref{hausdorffdistanceforsimilarity}, Theorem \ref{main} instead of Theorem \ref{main2}, and Proposition \ref{scalardisk2}.
\end{proof}
To conclude this paper we will briefly discuss closed similarity orbits of normal operators in von Neumann algebras.  We recall that \cite{Sh} completely classifies when two normal operators are approximately unitarily equivalent in von Neumann algebras.  Furthermore Theorem \ref{main} completely determines when one normal operator is in the closed similarity orbit of another normal operator in type III factors with separable predual.  Thus it is natural to ask whether a generalization of Theorem \ref{main} to type II factors may be obtained.
\par
Unfortunately the existence of a faithful, normal, tracial state on type II$_1$ factors inhibits when a normal operator can be in the closed similarity orbit of another normal operator.  Indeed suppose $\mathfrak{M}$ is a type II$_1$ factor and let $\tau$ be the faithful, normal, tracial state on $\mathfrak{M}$.  If $N, M \in \mathfrak{M}$ are such that $N \in \overline{\mathcal{S}(M)}$, it is trivial to verify that $\tau(p(N)) = \tau(p(M))$ for all polynomials $p$ in one variable.  In particular if $N, M \in \mathfrak{M}$ are self-adjoint and $N \in \overline{\mathcal{S}(M)}$ we obtain that $\tau(f(N)) = \tau(f(M))$ for all continuous functions on $\sigma(N) \cup \sigma(N)$ and, as $\tau$ is faithful and normal, this implies that $N$ and $M$ must have the same spectral distribution.  Therefore, if $N, M \in \mathfrak{M}$ are self-adjoint operators, $\sigma(M) = [0,\frac{1}{2}]$, and $\sigma(N) = [0,1]$, then, unlike in $\mathcal{B}(\mathcal{H})$, $N \notin \overline{\mathcal{S}(M)}$.  Combining the above arguments and {\cite[Theorem 1.3]{Sh}} we have the following result.
\begin{prop}
Let $\mathfrak{M}$ be a type II$_1$ factor.  If $N, M \in \mathfrak{M}$ are self-adjoint operators and $N \in \overline{\mathcal{S}(M)}$, then $N \sim_{au} M$.
\end{prop}

\end{document}